\documentclass[12pt,a4paper,reqno]{article}
\usepackage{amsmath,amsthm,amsfonts,amssymb,mathrsfs}
\usepackage[abbrev]{amsrefs}
\usepackage{color}
\usepackage{hyperref}
\usepackage[utf8]{inputenc}
\usepackage[T1]{fontenc}

\usepackage[normalem]{ulem}
\usepackage{color}
\definecolor{ashgrey}{rgb}{0.7, 0.75, 0.71}

\newtheorem{theorem}{Theorem}
\newtheorem{lemma}[theorem]{Lemma}

\newtheorem{prop}[theorem]{Proposition}
\newtheorem{cor}[theorem]{Corollary}

\theoremstyle{remark}
\newtheorem{rem}[theorem]{Remark}

\numberwithin{equation}{section}
\numberwithin{theorem}{section}

\def \bk {{\mathbf k}}
\def \RR {\mathbb{R}}
\def \NN {\mathbb{N}}

\def\({\left(}
\def\){\right)}
\def\[{\left[}
\def\]{\right]}

\def \ve {\varepsilon}

\newcommand{\fp}[2]{#1^{\left\langle  #2 \right\rangle}}

\allowdisplaybreaks

\begin{document}
\author{T. Jakubowski\footnote{Wroc{\l}aw University of Science and Technology, ul. Wybrze\.ze Wyspia\'nskiego 27, Wroc{\l}aw, Poland, email: tomasz.jakubowski@pwr.edu.pl},   G. Serafin\footnote{Wroc{\l}aw University of Science and Technology, ul. Wybrze\.ze Wyspia\'nskiego 27, Wroc{\l}aw, Poland, email: grzegorz.serafin@pwr.edu.pl}\\
%Faculty of Pure and Applied Mathematics,\\ 
}
\date{}
\title{Fractional Burgers equation\\ with singular initial condition 
\footnotetext{The first author was partially supported by Wroc\l{}aw University of Science and Technology grant 8211104160 MPK: 9130730000 and by the grant 2015/18/E/ST1/00239 of National Science Centre, Poland. The second author was partially supported by the grant 2015/18/E/ST1/00239 of National Science Centre, Poland.}}	
\maketitle

\begin{abstract}
We consider the fractional Burgers equation $ \Delta^{\alpha/2} u + b\cdot \nabla (u|u|^{(\alpha-1)/\beta})$ on $\RR^d$, $d\geq2$, with {$\alpha \in (1,2)$ and} $\beta>1$ and prove the existence of a solution for a large class of initial conditions, which contains functions that do not belong to any $L^p(\RR^d)$, $1\leq p\leq\infty$. Next, we apply the general results to the initial condition $u_0(x)=M|x|^{-\beta}$, $1<\beta<d$, and show the existence of a selfsimilar solution and derive its  properties such as smoothness, two-sided estimates, asymptotics and gradient estimates.
\end{abstract}

\section{Introduction}
Let $d \ge 2$ and $\alpha \in (1,2)$. Consider the fractional Burgers equation 
\begin{equation}\label{eq:problem}
\begin{cases}
u_t = \Delta^{\alpha/2} u + b\cdot \nabla (u|u|^{q}), &\qquad t>0,\\
u(0,\cdot)=u_0
\end{cases}
\end{equation}
in $\RR^d$, where  $ \beta>1$,  $b\in\RR^d$ and $q=\frac{\alpha-1}{\beta}$ are fixed. For technical reasons, we  assume (without loss of  generality) $b=(|b|,0,0,...,0)$. Here, $\Delta^{\alpha/2}$ is the fractional Laplacian 
\begin{align*}
\Delta^{\alpha/2} f(x) = \lim_{\ve\to 0^+} c_{d,\alpha} \int_{|y|>\ve} \frac{f(x+y)-f(x)}{|y|^{d+\alpha}} dy,\hspace{25pt} f\in C_c^{\infty}(\RR^d),
\end{align*}
where $c_{d,\alpha}$ is some constant. It may be also defined by the Fourier transform
\begin{align*}
\widehat{\Delta^{\alpha/2} f}(\xi) = -|\xi|^\alpha \hat{f}(\xi).
\end{align*}
We put
\begin{align*}
p(t,x) = \frac{1}{(2\pi)^d} \int_{\RR^d} e^{ix \cdot\xi} e^{-t|\xi|^\alpha}, \qquad t>0,\; x\in \RR^d.
\end{align*}
and denote $p(t,x,y) = p(t,y-x)$, $t>0,\ x,y\in\RR^d$. The function $p$ is the density of the semigroup 
$P_t = e^{t\Delta^{\alpha/2}}$
generated by the fractional Laplacian,
$$P_tf(x)=e^{t\Delta^{\alpha/2}}f(x)=\int_{\RR^d}p(t,x,y)f(y)dy.$$

By a solution to the Cauchy problem \eqref{eq:problem} we mean a mild solution, i.e. a function $u \colon \RR \times \RR^d \to \RR$ satisfying the Duhamel formula 
\begin{align}\label{eq:duhamel}
u(t,x) = P_tu_0(x)  + \int_0^t \int_{\RR^d} b \cdot \nabla_x p(t-s,x,z)  u(s,z)|u(s,z)|^{q}\,dz\,ds.
\end{align}

%Any potential  solution to this equation satisfies the scaling condition
%$$
%u(t,x) = \lambda^\beta u(\lambda^\alpha t, \lambda x).
%$$
%In particular, this follows
%$$
%||u(t,\cdot)||_p = t^{d/\alpha2p-\beta/2} ||u(1,\cdot)||_p,
%$$
%whenever $u(t,\cdot)\in L^p$, $p>1$. 

%\subsection{Main results \gs{Ja bym nie rozdzielal wstepu na 2 czesci}}
The fractional Burgers equation was intensely studied in recent years, see e.g. \cites{MR2659155, MR3789847, MR1637513, MR1849690, MR1881259, MR2407204, KMX, AIK, KNS, WW, MR3382709}. It is a generalization of the classical Burgers equation 
\begin{align*}
u_t = \nu u_{xx} - \tfrac{1}{2}(u^2)_x,
\end{align*}
which was introduced as a simplest model describing the turbulence phenomena (see e.g. \cite{Burgers}). The fractional counterpart \eqref{eq:problem} of this equation was studied for the first time by  Biler and Funaki in \cite{MR1637513}. Later on,  in \cite{MR1881259} Biler, Karch and Woyczyński considered some further generalizations, where the fractional Laplacian was replaced by some L\'evy operator and nonlinearity $f(u)$ was given by a smooth function $f$. {In particular, they proved that for $u_0\in L^1(\RR^d)\cap L^{\infty}(\RR^d)$ there exists a unique solution to \eqref{eq:problem} such that (see \cite[Theorem 3.1]{MR1881259})
\begin{align}\label{eq:BKWuin}
u\in C\big((0,\infty);W^{2,2}(\RR^d)\big)\cap C^1\big((0,\infty);L^2(\RR^d)\big),
\end{align}
}
and
\begin{align}\label{eq:BKWnorms}
\|u(t,\cdot)\|_\gamma\leq \|u_0\|_\gamma,\ \ \ \ \ t>0, \ \ \gamma\in[1,\infty].
\end{align}
We point out that in both papers the standing assumptions was $u_0 \in L^1(\RR^d)$. In \cite{MR1849690} the authors studied the equation \eqref{eq:problem} with the critical exponent $q=\frac{(\alpha-1)}d$, i.e.
\begin{align}\label{eq:critical}
	\begin{cases}
	u_t = \Delta^{\alpha/2} u + b\cdot \nabla u|u|^{(\alpha-1)/d}, &\qquad t>0,\\
	u(0,\cdot)=u_0.
\end{cases}
\end{align}
They showed that for $u_0= M\delta_0$, $M>0$, there is a unique selfsimilar solution (called source solution) to \eqref{eq:critical} satisfying the  scaling property $U(t,x) = t^{-d/\alpha}U(1,t^{-1/\alpha}x)$ (the same  as the density $p$). It turns out that this self-similar solution determines the long time behavior of solutions to a large class of Cauchy problems \eqref{eq:critical} with $u_0 \in L^1(\RR^d)$ and $\|u_0\|_1=M$ (see \cite[Theorem 2.2]{MR1849690}) 
\begin{align}\label{eq:assymp}
\lim_{t\to \infty}t^{\frac{d}{\alpha}(1-\frac{1}{\gamma})} \|u(t,\cdot) - U(t,\cdot))\|_\gamma = 0, \qquad \gamma \in [1,\infty].
\end{align}
If $u_0 \in L^1(\RR^d) \cap L^{\infty}(\RR^d)$, the existence of the solution to the problem \eqref{eq:critical} can by proved by using the Banach fixed point theorem (see \cite{MR1881259}). The case $u_0 = M\delta_0$ is in some sense similar to $u_0\in L^1(\RR^d)$  since such $u_0$ is still integrable and the solution may be constructed by using the approximation by the solutions with initial conditions from $L^1(\RR^d)$. Nevertheless, if $u_0$ does not belong to any $L^\gamma(\RR^d)$, $\gamma\in[1,\infty]$,  the construction by fixed point theorems fails.

In this paper we assume that the initial condition $u_0 \colon \RR^d \to \RR$ satisfies
\begin{align}\label{ass:A}
\sup_{ t>0,\, x\in\RR^d} t^{(\beta-1)/\alpha} \int_{\RR^{d}}|u_0(y)| p^{(d-1)}(t,\tilde x,\tilde y) d y\leq\mu, \tag{\textbf{A}}
\end{align}
where $\mu \in (0,\infty)$ is some constant. Here, $p^{(d-1)}$ is the density of $P_t$ in dimension $d-1$ and $\tilde{x} = {(x_2,\ldots, x_d)} \in \RR^{d-1}$ (see \eqref{def:pm} and Notation in Section 2). 
The main feature of the class of functions satisfying \eqref{ass:A} is   that it contains some functions not belonging to any $L^\gamma(\RR^d)$, $\gamma\in[1,\infty]$. This allows us to drop the common assumption $u_0\in L_1(\RR^d)$, or even $u_0\in L_1(\RR^d)\cap L^\infty(\RR^d)$, and consider some singular initial conditions.
Our first result is 
\begin{theorem}\label{thm:existence}
Let {$\alpha \in (1,2)$, $\beta>1$ and $q = (\alpha-1)/\beta$.} If $u_0$ satisfies \eqref{ass:A}, then there exists a solution $u(t,x)$ to the problem \eqref{eq:problem}. Furthermore, there is a constant $C>0$ depending {only} on $d,\alpha,\beta, b, \mu$ such that 
{
\begin{align}\label{eq:existence}
	 |u(t,x)| \le C\,t^{-1/\alpha}  \int_{\RR^d}     |u_0( y)| p^{(d-1)}(t, \tilde x, \tilde y) d y \le C\mu  t^{-\beta/\alpha}, \qquad t>0, \; x \in \RR^d.
\end{align}}
\end{theorem}

  In the proof of Theorem \ref{thm:existence}  we proceed similarly as in the case of source solutions and approximate the initial condition by the monotone sequence of functions from $L^1(\RR^d) \cap L^\infty(\RR^d)$. However, the convergence of the corresponding sequence of solutions is a major problem here. {We resolve it by showing that a certain functional acting on these solutions is uniformly bounded, which ensures the point-wise convergence of the sequence almost everywhere} (see Lemma \ref{lem:aux1}). This approach, however, requires the technical  assumptions $\beta>1$ and $d\geq2$.
  
The first bound in Theorem \ref{thm:existence} has one weakness - it does not depend on $x_1$, and therefore it is not optimal. Nevertheless, in many cases this inconvenience could be removed. We present it on an example, which  turns out to be quite special. Namely, in the second part of the paper we focus on  the solution {to the problem \eqref{eq:problem} with $u_0(x) = M|x|^{-\beta}$, $M>0$, $1<\beta<d$}, constructed in Theorem \ref{thm:existence}. Its first interesting property is selfsimilarity, in which one can see an analogy to the source solution in the case $\beta=d$ and $u_0=M\delta_0$. Next, we turn our attention to the two-sided pointwise estimates, which not only improve \eqref{eq:existence} in this case, but complement it with the lower bound of the same form. Such estimates were obtained for the source solution of \eqref{eq:critical}. In \cite{MR2407204} Karch and Brandolese showed that for $q=\frac{\alpha-1}{d}$ and  $u_0 = M\delta_0$ with $M>0$ sufficiently small, the solution {$u$}  to \eqref{eq:critical} admits the estimates $0 \le {u(t,x)} \le c p(t,x)$. In \cite{JS1} we generalized this result to any $M>0$ and obtained two-sided estimates $c^{-1} p(t,x) \le {u(t,x)} \le c p(t,x)$.
Furthermore, in \cite[Theorem 1.1]{JS2} we showed that for nonnegative $u_0\in L^1(\RR^d)\cap L^{\infty}(\RR^d)$ there is a constant $C>0$ such that the solution to \eqref{eq:problem} with $q\ge \frac{\alpha-1}{d}$ satisfies 
\begin{align}\label{eq:estJDE}
C^{-1}\,P_tu_0(x)\leq u(t,x)\leq CP_tu_0(x),\ \ \ \ \  \ t>0, \ x\in\RR^d.
\end{align}
Note that the pointwise estimates of  solutions are rather rare in the literature, as deriving them is more challenging than e.g. the $L^p$ bounds. Nevertheless,  they  give a better insight into the  behaviour of the solution. For instance the estimates \eqref{eq:estJDE} permitted the authors to improve the asymptotics given in \eqref{eq:assymp}.

In Theorem \ref{thm:sss} below we prove the existence  of the selfsimilar solution to the problem \eqref{eq:problem} for $u_0(x) = M|x|^{-\beta}$ with $M>0$, $1<\beta<d$ and provide its properties. 
Note that the assumption $\beta<d$ is required in order to get local integrability of $u_0$.
	%The problem studied in Theorem \ref{thm:sss} is also in some sense similar to the case of the source solutions mentioned above, which also admits the  selfsimilarity property. 

\begin{theorem}\label{thm:sss}
Let {$\alpha \in (1,2)$, $1< \beta < d$, $q = (\alpha-1)/\beta$, $M>0$} and $u_0(x)=M|x|^{-\beta}$. Then, there exists a function $U(x)$ such that 
\begin{enumerate}
\item $u(t,x) = t^{-\beta/\alpha}U(t^{-1/\alpha} x)$ is the solution to the problem \eqref{eq:problem}.
\item There exists a constant $C_1>1$  such that
\begin{align}\label{eq:sss2}
C_1^{-1}\frac{1}{(1+|x|)^\beta} \le U(x) \le  C_1\frac{1}{(1+|x|)^\beta}, \qquad  x\in\RR^d.
\end{align}
\item There is a constant $C_2$  such that for all $x\in\RR^d$ we have
\begin{align}
&\left|U(x)-P_1u_0(x)\right| \le C_2\frac{1}{(1+|x|)^{\beta+(\alpha-1)}}, \label{eq:sss3a} \\
&\left|U(x)-P_1u_0(x) - \alpha \int_0^{1} \int_{\RR^d} r^{d-\beta} \nabla_x p(1-r^\alpha,x,r w) [P_1u_0(w)]^{1+q} dw dr\right| \label{eq:sss3b}\\
&\qquad\qquad\qquad\qquad\qquad\qquad\qquad\qquad\qquad\qquad\qquad\qquad\le \frac{C_2}{(1+|x|)^{\beta+2(\alpha-1)}}. \notag
\end{align}
\item $U \in \mathcal C^{1}\(\RR^d\)$ and there is a constant $C_3$ such that 
\begin{align}\label{eq:sss4}
	|\nabla U(x)|\le C_3 \frac{1}{(1+|x|)^\beta}, \qquad  x\in\RR^d.
\end{align}
\end{enumerate}
All the constants $C_1,C_2, C_3$ depend only on $d,\alpha, \beta, M, b$ and might be calculated explicitly.
\end{theorem}
\begin{rem}
By Lemma \ref{lem:estpa}, the bounds in \eqref{eq:sss2} may be equivalently written in the form
\begin{align*}
u(t,x) \approx \frac{1}{(t^{1/\alpha}+|x|)^\beta} \approx P_t u_0(x), \qquad t>0,\; x\in\RR^d,
\end{align*}
which resembles both: \eqref{eq:estJDE} by representation involving the semi-group,  and the estimates of  the source solution in the critical case derived  in  \cite{JS2} by the exact form.
Similarly, since $\nabla_x u(t,x) = \lambda^{\beta+1}\nabla_x(\lambda^\alpha t,\lambda x)$, by \eqref{eq:sss4} we have
$$|\nabla_x u(t,x)|\lesssim \frac{1}{t^{1/\alpha}(t^{1/\alpha}+|x|)^\beta}, \qquad t>0,\; x\in\RR^d.$$
\end{rem}

The paper is organized as follows. In Section 2 we introduce the notation used in the paper and give some preliminary results on the density $p(t,x,y)$ and solutions to \eqref{eq:problem}. Section 3 is devoted to some general estimates and proof of Theorem \ref{thm:existence}. In Section 4 we prove Theorem \ref{thm:sss}.

%\del{For other developments of fractal Burgers equations see...}

\section{Preliminaries}
\subsection{Notation}
We will use the following notation. We write '':='' to indicate definitions, e.g. $a \wedge b := \min\{a,b\}$ and $a \vee b := \max\{a,b\}.$ {All the considered functions are tacitly assumed to be Borel measurable.} We write $f(x) \lesssim g(x)$ if $f, g \geq 0$ and there is a number $c>0$ not depending on other parameters than $d$, $\alpha$, $\beta$, $M$ and $b$ such that $f(x) \leq c g(x) $ for all arguments $x$. If $f \lesssim g$ and $g \lesssim f$, we write $f(x) \approx g(x)$. If the comparability constant depends additionally on some other parameters (e.g. on $\gamma$), we indicate the parameter over the comparison sign (e.g.  $f \stackrel{\gamma}{\approx} g$).

By $\nabla = (\partial_{x_1}, \ldots, \partial_{x_d})$ we denote the standard gradient operator, while $\nabla^\bk = \partial^{k_1}_{x_1} \ldots \partial^{k_d}_{x_d}$, where $\bk = (k_1, \ldots,k_d) \in \NN_0^d$. The norm of the multi-index $\bk$ is given by $|\bk| = k_1+\ldots+k_d$. By writing $\nabla_x$ and $\nabla^\bk_x$ we indicate  variable with respect to which the derivatives are taken.

For any $x =(x_1,x_2,\ldots,x_d)\in \RR^d$ we denote $\tilde x:=(x_2,x_3,...,x_d)\in\RR^{d-1}$. For any function $f\colon(0,\infty)\times\RR^d\rightarrow\RR$, we define its rescaled version by
\begin{align}\label{def:star}
	f^*(t,x):=t^{\beta/\alpha}f(t,t^{1/\alpha}x).
\end{align}
Similarly, we put
\begin{align*}
P_t^*f(x) := t^{\beta/\alpha} P_t f(t^{1/\alpha}x) = t^{\beta/\alpha} \int_{\RR^d} p(t,t^{1/\alpha}x,y) f(y) dy.
\end{align*}
{For any $a \in \RR$ and $\gamma>1$ we define the french power $\fp{a}{\gamma} = a |a|^{\gamma-1}$. In particular, $|u(s,z)|u(s,z)^q = \fp{u(s,z)}{q+1}$.}
\subsection{Fractional Laplacian semigroup }
Although in the paper we generally assume  $\alpha \in (1,2)$ and $d \ge2$, the results of this section are valid for $\alpha \in(0,2)$ and $d\ge1$. The methodology we propose in the next sections requires consideration of  the density of the semigroup $P_t$ in various dimensions at the same time. Therefore,  for  $m\in\{1,2,3,\ldots\}$ we put 
%\gs{TO JUZ WLASCIWIE BYLO WE WSTEPIE. MOZNA TAM TYLKO NAPISAC, ZE $p$ TO GESTOSC POLGRUPY}
\begin{align}\label{def:pm}
p^{(m)}(t,x) = \frac{1}{(2\pi)^m} \int_{\RR^m} e^{ix \cdot\xi} e^{-t|\xi|^\alpha}, \qquad t>0,\; x\in \RR^m,
\end{align}
and denote $p^{(m)}(t,x,y) = p^{(m)}(t,y-x)$, $t>0,\,x,y\in\RR^m$.  
In the case  $m=d$ we simply write $p^{(d)}(t,x,y) = p(t,x,y) = p(t,y-x)$.
Below, we collect some basic properties of the density $p$.
The following scaling property holds
\begin{align}\label{eq:pscaling}
p(\lambda^{\alpha} t,\lambda x,\lambda y)=\lambda^{-d}p(t,x,y), \qquad \lambda>0.
\end{align}
Furthermore, the function $p^{(m)}$ admits the two-sided estimates 
\begin{align}\label{eq:pest}
p^{{(m)}}(t,x,y)\approx \frac{t}{\(t^{1/\alpha}+|x-y|\)^{{m}+\alpha}}, \qquad t>0,\, x,y\in\RR^m, 
\end{align}
where the constants in the bounds depend on $m$ and $\alpha$. In particular, this implies
\begin{align}\label{eq:Lpp}
\|p(t,\cdot)\|_\gamma &=  t^{-\frac d\alpha (1-1/\gamma)}\|p(1,\cdot)\|_\gamma, \qquad t>0,\\
p(t,x,y) &\lesssim t^{-1/\alpha}p^{(d-1)}(t,\tilde x,\tilde y),  \hspace{29pt}  t>0,\,x,y,\in\RR^d. \label{eq:p<p(d-1)}
\end{align}
\begin{lemma}
For $t>0$, $\tilde{x} \in \RR^{d-1}$, $y\in \RR^d$, we have
\begin{align}\label{eq:intp}
\int_{\RR}p(t,x,y)dx_1=p^{(d-1)}(t,\tilde x, \tilde y).
\end{align}
\end{lemma}
\begin{proof}
By the subordination formula (see \cite{MK}) there is a nonnegative function $\eta_t(s)$ such that 
\begin{align}\label{eq:subordination}
p^{(m)}(t,x,y) = \int_0^\infty (4\pi s)^{-m/2} e^{-|x-y|^2/4s} \eta_t(s) ds, \qquad m \in \NN, t>0, \,x,y\in\RR^m.
\end{align}
Hence,
\begin{align*}
\int_{\RR}p(t,x,y)dx_1&=\int_0^\infty \int_{\RR}(4\pi s)^{-d/2} e^{-|x-y|^2/4s} \eta_t(s) dx_1 ds \\
&= \int_0^\infty (4\pi s)^{-(d-1)/2} e^{-|\tilde x-\tilde y|^2/4s} \eta_t(s) ds =  p^{(d-1)}(t,\tilde x, \tilde y).
\end{align*}
\end{proof}
The subordination formula \eqref{eq:subordination} allows us also to derive the following \mbox{estimates of $\nabla^\bk p$.}
\begin{lemma} For any $\bk\in \NN_0^d$, $t>0$ and $x\in\RR^d$, we have
\begin{align}\label{eq:gradp}
|\nabla^{\bk}_x p(t,x)|\stackrel{\bk}{\lesssim}  \frac{p(t,x)}{(t^{1/\alpha}+|x|)^{|\bk|}}.
\end{align}
\end{lemma}
\begin{proof}
Let
$\bk=(k_1,\ldots,k_d) \in \NN_0^d$ and 
 $g(s,x) = (4\pi s)^{-d/2} e^{-|x|^2/(4s)}$ be the Gaussian kernel. Then,
\begin{align*}
\nabla^{\bk} g(s,x) = (4s)^{-|\bk|/2} g(s,x) \prod_{i=1}^d H_{k_i}(x_i/2\sqrt s), 
\end{align*}
where $H_n$ are the Hermite polynomials. By subordination formula \eqref{eq:subordination} and the Fubini theorem, we have 
\begin{align*}
\nabla^{\bk} p(t,x) &=  \int_0^\infty (4s)^{-|\bk|/2} g(s,x) \prod_{i=1}^d H_{k_i}(x_i/2\sqrt s) \eta_t(s)  \,ds.
%\\
%&= \prod_{i=1}^d h_{k_i}(x_i) \int_0^\infty \pi^{|\bk|} (4\pi s)^{-(d+|\bk|)/2} e^{-|\hat{x}|^2/(4s)} \eta_t(s)  \,ds \\
%&=    \pi^{|\bk|} p^{(d+2|\bk|)}(t,\hat{x}) \prod_{i=1}^d h_{k_i}(x_i),
\end{align*}
Since every $H_{k_i}$ is a polynomial of degree $k_i$, we bound
$$\left|\prod_{i=1}^d H_{k_i}(x_i/2\sqrt s)\right|\stackrel{\bk}{\lesssim} \prod_{i=1}^d [1+|x_i/2\sqrt s|^{k_1}]\lesssim 1+|x/2\sqrt s|^{|\bk|}.$$
Consequently
\begin{align*}
|\nabla^{\bk} p(t,x) |&\stackrel{\bk}{\lesssim}  \int_0^\infty (4s)^{-|\bk|/2} g(s,x){\eta_t(s)}\,ds+|x|^{|\bk|}\int_0^\infty (4s)^{-|\bk|} g(s,x){\eta_t(s)}\,ds\\
&=p^{(d+|\bk|)}(t,x_+)+|x|^{|\bk|}p^{(d+2|\bk|)}(t,x_{++}),
\end{align*}
where ${x_{+}} = (x_1,\ldots,x_d,0,\ldots,0) \in \RR^{d+|\bk|}$ and ${x_{++}} = (x_1,\ldots,x_d,0,\ldots,0) \in \RR^{d+2|\bk|}$. Eventually, by \eqref{eq:pest}, we obtain
\begin{align*}
|\nabla^{\bk} p(t,x)| \stackrel{\bk}{\lesssim} \frac{ t}{(t^{1/\alpha} + |x|)^{d+\alpha+|\bk|}}+\frac{|x|^{|\bk|} t}{(t^{1/\alpha} + |x|)^{d+\alpha+2|\bk|}} \approx \frac{1}{(t^{1/\alpha} + |x|)^{|\bk|}} p(t,x),
\end{align*}
as required.
\end{proof}

The next lemma describes the behaviour of  $P_t$ acting on the function $(r^{1/\alpha}+|y|)^{-\gamma}$, which, as we will show, is comparable with a solution to \eqref{eq:problem} for $u_0(x)=M|x|^{-\gamma}$.
\begin{lemma}  \label{lem:estpa}
Let  $0<\gamma<d$.  For $x\in\RR^d$, $t >0$ and $r\ge0$  we have 
\begin{align}\label{eq:estpa}
\int_{\RR^d} p(t,x,y) (r^{1/\alpha}+|y|)^{-\gamma} dy\stackrel{\gamma}{\approx}\frac{1}{(r^{1/\alpha}+t^{1/\alpha}+|x|)^{\gamma}}.  
\end{align}
\end{lemma}
\begin{proof}
By \cite[Lemma 2.3]{MR3933622},
\begin{align}\label{eq1:estpa}
\int_{\RR^d} p(t,x,y)|y|^{-\gamma} dy\stackrel{\gamma}{\approx} \frac{1}{(t^{1/\alpha}+|x|)^{\gamma}},
\end{align}
which proves \eqref{eq:estpa} for $r=0$. Now, let $r>0$. By \eqref{eq1:estpa}, we get
\begin{align*}
\int_{\RR^d} p(t,x,y)(r^{1/\alpha}+|y|)^{-\gamma} dy &\approx  \int_{\RR^d} \int_{\RR^d} p(t,x,y) p(r,y,z) |z|^{-\gamma} dz dy \\
&= \int_{\RR^d} p(t+r,x,z) |z|^{-\gamma} dz  \approx \frac{1}{((r+t)^{1/\alpha}+|x|)^{\gamma}}.
\end{align*}
Since $(r+t)^{1/\alpha} \approx r^{1/\alpha} + t^{1/\alpha}$, we obtain \eqref{eq:estpa}.
\end{proof}
%TEN WZOR PONIZEJ RZECZYWISCIE WYKORZYSTUJEMY 2 RAZY, ALE ON MI NIE PASUJE W TYM ROZDZIALE, JEST ZBYT BYLEJAKI. DLA MNIE MOZNA NAWET 2 RAZY POWTORZYC TEN DOWOD, TO TYLKO SKALOWANIE TAK NAPRAWDE
%\gs{At the end of this section we note that the semi-group and scaling properties of $p(t,x,y)$ imply
%\begin{align}\label{eq:CKtmp}
%\int_{\RR^d} p(1-r^\alpha,x,rw) p(s^\alpha,sw,y) dw = p(s^{\alpha},sx, ry), \qquad x,y,\in\RR^d,
%\end{align}
%which will be used in the sequel. Indeed, we have
%\begin{align*}
%\int_{\RR^d} p(1-r^\alpha,x,rw) p(s^\alpha,sw,y) dw & = \int_{\RR^d} (rs)^{-d}p(r^{-\alpha}-1,r^{-1}x,w) p(1,w,s^{-1}y) dw \\
%&=  (rs)^{-d}p(r^{-\alpha},r^{-1}x, s^{-1}y) = r^{-d}p((s/r)^{\alpha},(s/r)x, y)\\
%&= p(s^{\alpha},sx, ry).
%\end{align*}
%}

\subsection{Basic properties of solutions to \eqref{eq:problem}}
Let us observe that
 if $u$ and $v$ are the solutions to \ref{eq:problem} with initials conditions $u_0$ and $-u_0$, respectively, then $u(t,x) = -v(t,x)$ for all $t>0$ and $x\in\RR^d$.

As mentioned in Introduction, if $u_0\in L^1(\RR^d)\cap L^{\infty}(\RR^d)$ then there exists a unique solution $u(t,x)$ to the problem \eqref{eq:problem}. Furthermore, \cite[Corollary 3.1]{MR1881259} implies the monotonicity property of solutions. Namely, if $u$ and $v$ are the solutions of (1.1) with initials conditions
$u_0, v_0\in L^1(\RR^d)\cap L^{\infty}(\RR^d)  $ and $u_0 \geq v_0$, then
\begin{align}\label{eq:monotonicity}
u(t, x) \geq v(t, x),\qquad t > 0, x\in \RR^d. 
\end{align}
Let us now denote by $\overline u$ and $\underline u$ the solutions to \ref{eq:problem} with initials conditions $|u_0|$ and $-|u_0|$, respectively. Since $\overline u=-\underline u$, we get by \eqref{eq:monotonicity}
\begin{align}\label{eq:|u|<baru}
|u|\leq |\overline u|\vee|\underline u|=\overline u.
\end{align}

For technical reasons we will often consider the rescaled version of solutions. 
In order to represent $u^*$ (cf. \eqref{def:star}) in a recursive manner as in \eqref{eq:duhamel}, we rewrite the integral therein as follows.
\begin{align*}
& \int_0^t \int_{\RR^d} b \cdot \nabla_z p(t-s,t^{1/\alpha}x,z)  \fp{u(s,z)}{q+1}\,dz\,ds   \\
& = \int_0^t \int_{\RR^d} b \cdot \nabla_z p(t-s,t^{1/\alpha}x,s^{1/\alpha}w)  s^{d/\alpha}  \fp{u(s,s^{1/\alpha} w)}{q+1}\,dw\,ds  \\
& = \int_0^1 \int_{\RR^d}  b \cdot \nabla_w p(t-st,t^{1/\alpha}x,(st)^{1/\alpha} w)  t(st)^{d/\alpha}  \fp{u(st,(st)^{1/\alpha}w)}{q+1}\,dw\,ds  \\
& = \int_0^1 \int_{\RR^d}  b \cdot \nabla_w p(1-s,x,s^{1/\alpha} w) t^{1-\frac1\alpha} s^{d/\alpha}  \fp{u(st,(st)^{1/\alpha}w)}{q+1}\,dw\,ds  \\
& = \alpha \int_0^1 \int_{\RR^d} t^{1-\frac1\alpha} r^{d+\alpha-1} b \cdot\nabla_w p(1-r^\alpha,x,r w)   \fp{u(r^\alpha t,rt^{1/\alpha}w)}{q+1}\,dw\,dr \\
&=\alpha \int_0^1 \int_{\RR^d}  r^{d-\beta} b \cdot\nabla_w p(1-r^\alpha,x,r w)   \fp{u^*(r^\alpha t,w)}{q+1}\,dw\,dr.
\end{align*}
Thus, we have
\begin{align}\label{eq:duhrescalled}
 u^*(t,x)
=P_t^*u_0(x)+\alpha \int_0^1 \int_{\RR^d}  r^{d-\beta} b \cdot \nabla_x p(1-r^\alpha,x,r w)   \fp{u^*(r^\alpha t,w)}{q+1}\,dw\,dr.
\end{align}

\section{General results}

\subsection{An integral conservation law} 
Up till the end of the paper we assume that $\alpha \in (1,2)$ and $d \ge2$.
Let $F:\RR^{d-1}\rightarrow \RR$ be any function. 
For a function $h:\RR^d\rightarrow \RR$ we denote 
$$h_F:=\int_{\RR^d}h(x)F(\tilde x)dx,$$
whenever the integral is convergent.
%where
%$$\RR^d\ni x\rightarrow \tilde x=(x_2,x_3,...,x_d)\in\RR^{d-1}.$$

\begin{lemma}\label{lem:weightedineq}
Let $F:\RR^{d-1}\rightarrow \RR$.  Then, for any $u$ satisfying \eqref{eq:duhamel} with {$|(P_tu_0)_F|<\infty$} for every $t>0$ and such that 
\begin{align}\label{eq:intcondition}
\int_0^t \int_{\RR^d} \int_{\RR^d} |\nabla_x p(t-s,x,z)|  |u(s,z)|^{q+1}|F(\tilde{x})|\, dx\,dz\,ds <\infty,
\end{align}
\ we have
\begin{align*}
u(t,\cdot)_F = (P_tu_0)_F, \qquad t>0.
\end{align*}
%provided $(P_tu_0)_F <\infty.$ TO ZALOZENIE JEST NIEPOTRZEBNE, JESLI CALKA JEST SKONCZONA \tj{MI SIE WYDAJE, ŻE TRZEBA COŚ TAKIEGO ZALOŻYĆ, MOŻE NAWET LEPIEJ NAPISAĆ $|P_tu_0|_F<\infty$}
\end{lemma}
\begin{proof}
Since $b = (|b|,0, \ldots, 0)$, by \eqref{eq:duhamel} and Fubini theorem, we have 
\begin{align*} 
u(t,\cdot)_F&=\int_{\RR^d}\(P_tu_0(x) + \int_0^t \int_{\RR^d} b \cdot \nabla_x p(t-s,x,z)  \fp{u(s,z)}{q+1} \,dz\,ds\)F(\tilde{x})dx\\
&=(P_tu_0)_F +\int_0^t \int_{\RR^d}\( \int_{\RR^d} |b| \partial_{x_1} p(t-s,x,z) F(\tilde{x})dx \)  \fp{u(s,z)}{q+1}\,dz\,ds\\
&=(P_tu_0)_F,
\end{align*}
where we used \eqref{eq:pest} and the fact that $F(\tilde x)$ does not depend on $x_1$.
\end{proof}

In particular, we show that the assumptions of Lemma \ref{lem:weightedineq} are satisfied if  $u_0\in L^1(\RR^d)\cap L^{\infty}(\RR^d)$ and $F$ is bounded.

\begin{cor}\label{lem:mcu0}
Suppose  $u$ is a solution to \eqref{eq:problem} with   $u_0\in L^1(\RR^d)\cap L^{\infty}(\RR^d)$. Let $F \colon \RR^{d-1} \to \RR$  be bounded. Then, 
\begin{align*}
\int_{\RR^d} F(\tilde{w}) u(t,w)\, dw = \int_{\RR^d} F(\tilde{w}) P_tu_0(w)\, dw, \qquad t>0.
\end{align*}
\end{cor}
\begin{proof} 
Let $K = \sup_{\tilde{w} \in \RR^{d-1}} |F(\tilde{w})|$. Since $|P_t u_0|_F\le K\|u_0\|_1<\infty$, by Lemma \ref{lem:weightedineq}, we only need to show that the condition \eqref{eq:intcondition} holds. By \eqref{eq:BKWnorms},  \eqref{eq:gradp} and \eqref{eq:intp} we have
\begin{align*}
&\int_0^{t}\int_{\RR^d}\int_{\RR^d}  |\nabla p(t-s,w,z)||u(s,z)|^{1+q} |F(\tilde  w)| dw\,dz\,ds \\
&\lesssim K \int_0^{t}\|u_0\|^q_{\infty}\int_{\RR^d}\int_{\RR^d}  \frac{p(t-s,w,z)}{(t-s)^{1/\alpha}}|u(s,z)|dw\,dz\,ds \\
&=K\|u_0\|^q_{\infty} \int_0^{t}\frac{1}{(t-s)^{1/\alpha}}\int_{\RR^d}  |u(s,z)|dz\,ds = K\|u_0\|^q_{\infty}\|u_0\|_1 \frac{\alpha t^{1-1/\alpha}}{\alpha-1}<\infty, 
\end{align*}
which ends the proof.   
\end{proof}
If in Corollary \ref{lem:mcu0} we take $F(\tilde{w}) =p^{(d-1)}(t_0,\tilde{x},\tilde{w})$ with $t_0>0$ and $x \in \RR^d$, by \eqref{eq:pest} we get
\begin{align}\label{eq:mcu0}
\int_{\RR^d} p^{(d-1)}(t_0,\tilde{x},\tilde{w}) u(t,w)\, dw = \int_{\RR^d} p^{(d-1)}(t_0,\tilde{x},\tilde{w}) P_tu_0(w)\, dw, \qquad t>0.
\end{align}

The following lemma will be frequently exploit in the sequel. It also partially explains the introduction of the assumption \eqref{ass:A}.
\begin{lemma}\label{lem:aux1} Suppose  $u$ is a solution to \eqref{eq:problem} with   $u_0\in L^1(\RR^d)\cap L^{\infty}(\RR^d)$. 
%and satisfies \eqref{eq:duhamel} and
%\begin{align*}
%\int_0^{t}\int_{\RR^d}\frac{p^{(d-1)}(1, x,z)}{\( t-s\)^{1/\alpha}}|u(s,z)|^{1+q}\,dz\,ds<\infty
%\end{align*}
%for any  $x\in\RR^d$ and $t>0$. 
Then, for $x\in\RR^d$, $t>0$ and $r\in(0,1)$  we have
\begin{align}\label{aux1}
\int_{\RR^d} p(1-r^\alpha,x,r w) \left|u^*(r^\alpha t,w)\right|\,dw\lesssim \frac{ r^{\beta-d}}{(1-r)^{1/\alpha}}\,t^{(\beta-1)/\alpha} \int_{\RR^{d}} p^{(d-1)}(t,t^{1/\alpha}\tilde x,\tilde y) |u_0( y)| d y.
\end{align}
\end{lemma}

\begin{proof}
In view of \eqref{eq:|u|<baru}, it is enough to consider $u_0\geq0$. In that case, $u\geq0$ and  by \eqref{eq:p<p(d-1)}, \eqref{eq:pscaling} and \eqref{eq:mcu0} we get 
    \begin{align*}
 &\int_{\RR^d} p(1-r^\alpha,x,r w) u^*(r^\alpha t,w)\,dw\\
 &\lesssim  (1-r^\alpha)^{-1/\alpha} \int_{\RR^d} p^{(d-1)}(1-r^\alpha,\tilde x,r\tilde  w) \({r^\alpha t}\)^{\beta/\alpha} u(r^\alpha t,rt^{1/\alpha}w)\,dw\\
 &=(1-r^\alpha)^{-1/\alpha} (r^\alpha t)^{(\beta-d)/\alpha}t^{d/\alpha}\int_{\RR^d} p^{(d-1)}(t-tr^\alpha,t^{1/\alpha} \tilde x,\tilde  w)  u(r^\alpha t,w)\,dw.\\
 &\lesssim   (1-r)^{-1/\alpha}\int_{\RR^d} p^{(d-1)}(1-r^\alpha,\tilde x,r\tilde  w)\({r^\alpha t}\)^{\beta/\alpha}P_{r^\alpha t}u_0(rt^{1/\alpha}w)\,dw\\
  &=   (1-r)^{-1/\alpha}\int_{\RR^d} p^{(d-1)}(1-r^\alpha,\tilde x,r\tilde w)\int_{\RR^d}\({r^\alpha t}\)^{\beta/\alpha}p(r^\alpha t,r t^{1/\alpha}w,y) u_0(y)dy\,d w\\
    &=  (1-r)^{-1/\alpha} \int_{\RR^{d-1}}p^{(d-1)}(1-r^\alpha,\tilde x,r\tilde  w)\int_{\RR^{d}}\({r^\alpha t}\)^{(\beta-1)/\alpha}p^{(d-1)}(r^\alpha t,r t^{1/\alpha}\tilde w,\tilde y) u_0(y)dy\,d \tilde w,
    \end{align*}
    where, to obtain the last equality, we used \eqref{eq:intp} and performed the integration with respect to $w_1$. 
    Finally,  using the Tonelli theorem, \eqref{eq:pscaling}, and the semi-group properties of $p$, we arrive at
       \begin{align*}
     &\int_{\RR^d} p(1-r^\alpha,x,r w) u^*(r^\alpha t,w)\,dw\\
    &\lesssim  \frac{r^{\beta-2d+1} t^{(\beta-d)/\alpha}}{(1-r)^{1/\alpha}} \int_{\RR^{d}}u_0(y)\int_{\RR^{d-1}}p^{(d-1)}(r^{-\alpha}-1,r^{-1}\tilde x,\tilde  w)p^{(d-1)}(1,\tilde w,(r t^{1/\alpha})^{-1}\tilde y)\,d\tilde w\, dy\\
&=\frac{r^{\beta-2d+1} t^{(\beta-d)/\alpha}}{(1-r)^{1/\alpha}} \int_{\RR^{d}}u_0(y)p^{(d-1)}(r^{-\alpha},r^{-1}\tilde x,(r t^{1/\alpha})^{-1}\tilde y)\,d y\\
&=\frac{{r^{\beta-d} }t^{(\beta-1)/\alpha}}{(1-r)^{1/\alpha}} \int_{\RR^{d}}u_0(y)p^{(d-1)}(t,t^{1/\alpha}\tilde x,\tilde y)\,d y,
    \end{align*}
    which ends the proof.
%    \begin{align*}\nonumber
%    \int_{\RR^d} p(1-r^\alpha,x,r w) u^*(r^\alpha t,w)\,dw
%      \lesssim  \frac{r^{\beta-d}}{(1-r)^{1/\alpha}}t^{(\beta-1)/\alpha} \int_{\RR^{d}}u_0( y) p^{(d-1)}(t,t^{1/\alpha}\tilde x,\tilde y) d y.
%  \end{align*}
\end{proof}
\subsection{Proof of Theorem \ref{thm:existence}}
\begin{theorem}\label{thm:tildeunorm} 
Suppose  $u$ is a solution to the problem \eqref{eq:problem}. If the assumption \eqref{ass:A} is satisfied and 
$$\sup_{t>0,\,x\in\RR^d} |u^*(t,x)|<\infty,$$ 
then  there exists a constant $C>0$ depending only on $\alpha,  d$ such that
$$\sup_{t>0,\,x\in\RR^d} |u^*(t,x)| \leq 1\vee \big(C(\mu+1)(|b|+1)\big)^{\frac{\beta}{(\alpha-1)(\beta-1)}}.$$ 
\end{theorem}
\begin{proof}
Denote $\kappa = \sup\limits_{t>0,\,x\in\RR^d} |u^*(t,x)|$ and assume $\kappa> 1$. By \eqref{eq:duhrescalled}, we have
\begin{align*}
&| u^*(t,x)| \leq \|P^*_t|u_0|\|_\infty\\
& +C|b|\(\int_0^{1-1/2\kappa^\alpha}+\int_{1-1/2\kappa^\alpha}^1\) 
\int_{\RR^d} \frac{r^{d-\beta} }{(1-r^\alpha)^{1/\alpha}}p(1-r,x,r w) | u^*(r^\alpha t,w)|^{(\alpha-1+\beta)/\beta}\,dw\,dr\\
&=:\|P^*_t|u_0|\|_\infty+C|b|\(I_1+I_2\),
\end{align*}
for some $C=C(\alpha, d)$. First, by the inequality \eqref{eq:p<p(d-1)} and the assumption \eqref{ass:A} we have
\begin{align*}
P^*_t|u_0|(x)=t^{\beta/\alpha}\int_{\RR^d}p(t,t^{1/\alpha}x,y)|u_0(y)|dy\lesssim t^{(\beta-1)/\alpha}\int_{\RR^d}p^{(d-1)}(t,t^{1/\alpha}\tilde x,\tilde y)|u_0(y)|dy\leq \mu.
\end{align*}
Let us pass to estimating the integral $I_2$:
\begin{align*}
I_2&\leq \kappa^{(\alpha-1+\beta)/\beta}\int_{1-1/2\kappa^\alpha}^1 r^{d-\beta}(1-r)^{-1/\alpha} \int_{\RR^d}  p(1-r,x,r w)\,dw\,dr\\
&= \kappa^{(\alpha-1+\beta)/\beta}\int_{1-1/2\kappa^\alpha}^1 r^{-\beta}(1-r)^{-1/\alpha} dr.
\end{align*}
Then, since $\kappa> 1$, we have $1-1/2\kappa^\alpha>1/2$, and consequently
\begin{align*}
I_2&\leq \kappa^{(\alpha-1+\beta)/\beta}2^\beta\int_{1-1/2\kappa^\alpha}^1 (1-r)^{-1/\alpha}dr\\
&= \kappa^{(\alpha-1+\beta)/\beta}2^\beta\frac{\alpha}{\alpha-1} \(\frac1{2\kappa^\alpha}\)^{1-1/\alpha} 
= \kappa^{1-(\beta-1)\frac{(\alpha-1)}{\beta}}\,2^{\beta-1+1/\alpha}\frac{\alpha }{\alpha-1} .
\end{align*}
Concerning the integral $I_1$, we have
\begin{align}\label{eq:aux2}
I_1\leq \kappa^{\frac{\alpha-1}{\beta}} \int_0^{1-1/2\kappa^\alpha}r^{d-\beta} (1-r)^{-1/\alpha}\int_{\RR^d} p(1-r,x,r w) |u^*(r^\alpha t,w)|\,dw\,dr.
\end{align}
Applying \eqref{aux1} to \eqref{eq:aux2} and taking advantage of the assumption \eqref{ass:A}, we eventually obtain
\begin{align*}
I_1\lesssim \mu\kappa^{\frac{\alpha-1}{\beta}} \int_0^{1-1/2\kappa^\alpha}(1-r)^{-2/\alpha}\,dr=\mu \, 2^{2/\alpha-1}\kappa^{1-(\beta-1)\frac{(\alpha-1)}{\beta}},
\end{align*}
which leads to 
$$\kappa\lesssim \mu+|b|\(1+\mu\)\(\kappa^{1-(\beta-1)\frac{(\alpha-1)}{\beta}}\).$$
Dividing both sides by $\kappa^{1-(\beta-1)\frac{(\alpha-1)}{\beta}}$ and keeping in mind the assumptions $\alpha, \beta>1$ and $\kappa>1$, we obtain
$$ \kappa^{(\alpha-1)(\beta-1)/\beta}\lesssim \mu+|b|\(1+\mu\)\leq  (1+|b|)(\mu+1),$$
which implies  the bound in the assertion of the theorem.
\end{proof}
%Removing the $^*$-notation, Theorem \ref{thm:tildeunorm}  is equivalent to
%\begin{cor}\label{cor:u<} Under the assumption of the Theorem \ref{thm:tildeunorm}, there exists $C=C(d,\alpha,\beta,|b|, \mu)$ such that
%$$\|u(t,\cdot)\|_\infty\leq Ct^{-\beta/\alpha}.$$
%\end{cor}
Next we provide an improved upper bound, which depends on the space arguments as well.

\begin{prop}\label{thm:genest} Under the assumptions of Theorem \ref{thm:tildeunorm}, we have
\begin{align}\label{eq:genest}
	|u(t,x)|\leq C\,t^{-1/\alpha}  \int_{\RR^d}     |u_0( y)| p^{(d-1)}(t, \tilde x, \tilde y) d y
\end{align}
for some $C=C(\alpha,\beta,d, |b|, \mu)$.
\end{prop}
\begin{proof} 
As in the proof of the previous theorem, we denote $\kappa = \sup\limits_{t>0,\,x\in\RR^d} |u^*(t,x)$|. Let $0 < \ve < 1/2$  and put
\begin{align*}
H_\ve(t,x) &= \int_0^{1-\varepsilon} (1-r^\alpha)^{-1/\alpha} r^{d-\beta} \int_{\RR^d}   p(1-r^\alpha,x,r w)    |u^*(r^\alpha t,w)|\,dw\,dr,\\
h_\ve(t,x) &= \int_{1-\varepsilon}^1 (1-r^\alpha)^{-1/\alpha} r^{d-\beta} \int_{\RR^d}   p(1-r^\alpha,x,r w)    |u^*(r^\alpha t,w)|\,dw\,dr.
\end{align*}
We also let $H(t,x) = H_\ve(t,x) + h_\ve(t,x)$. By Theorem \ref{thm:tildeunorm} and the Duhamel formula \eqref{eq:duhrescalled}, we have
\begin{align}\label{eq1:u<}
|u^*(t,x)| \le P_t^* |u_0|(x) + c\kappa^q H(t,x). 
\end{align}
By \eqref{eq:p<p(d-1)},
\begin{align*}
P^*_t|u_0|(x)&=t^{\beta/\alpha}\int_{\RR^d}|u_0(y)|p(t,t^{1/\alpha}x,y)dy\lesssim  t^{(\beta-1)/\alpha}\int_{\RR^d}|u_0(y)|p^{(d-1)}(t,t^{1/\alpha}\tilde x,\tilde y)dy.
\end{align*}
Next, applying \eqref{aux1} we get 
\begin{align*}
H_\varepsilon(t,x)&\lesssim \int_0^{1-\varepsilon}(1-r)^{-2/\alpha} t^{(\beta-1)/\alpha} \int_{\RR^{d}}|u_0( y)| p^{(d-1)}(t,t^{1/\alpha}\tilde x,\tilde y) d y\,dr\\
&=\frac{(2-\alpha)}{\alpha}\varepsilon^{1-2/\alpha}\, t^{(\beta-1)/\alpha} \int_{\RR^{d}}|u_0( y)| p^{(d-1)}(t,t^{1/\alpha}\tilde x,\tilde y) \,dy.
\end{align*}
Hence, it is enough to show that for some $\ve>0$,
\begin{align}\label{eq0:genest}
	h_\ve(t,x) \leq C \left(P_t^*|u_0|(x) +  H_\ve(t,x)\right)
\end{align}
holds {with} some $C=C(\alpha,\beta,d, |b|, \mu,\ve)$. By virtue of \eqref{eq:p<p(d-1)} {and \eqref{eq1:u<}},
\begin{align*}
h_\ve(t,x) &\le \int_{1-\varepsilon}^1 (1-r^\alpha)^{-1/\alpha} r^{d-\beta} \int_{\RR^d}   p(1-r^\alpha,x,r w)  \Big(P_{r^\alpha t}^* |u_0|(w) + c\kappa^q H(r^\alpha t,w)\Big)\,dw\,dr\\ 
&=: I_1 + c\kappa^qI_2.
\end{align*}
Now, by the Tonelli theorem, \eqref{eq:pscaling} and the semigroup  properties of $p$,
\begin{align*}
I_1 &= \int_{1-\varepsilon}^1 (1-r^\alpha)^{-1/\alpha} r^{d-\beta} \int_{\RR^d}   p(1-r^\alpha,x,r w)  P_{r^\alpha t}^* |u_0(w)| \,dw\,dr \\
&= \int_{1-\varepsilon}^1 (1-r^\alpha)^{-1/\alpha} r^{d}t^{\beta/\alpha} \int_{\RR^d}   p(1-r^\alpha,x,r w)  \int_{\RR^d}p(r^\alpha t,rt^{1/\alpha}w,y)|u_0(y)|\,dy \,dw\,dr \\ 
&= \int_{1-\varepsilon}^1 (1-r^\alpha)^{-1/\alpha} r^{-d}t^{(\beta-d)/\alpha} \int_{\RR^d} |u_0(y)| \int_{\RR^d} p(r^{-\alpha}-1,r^{-1}x, w) p(1,w,(rt^{1/\alpha})^{-1}y)\,dw \,dy\,dr \\ 
&= \int_{1-\varepsilon}^1 (1-r^\alpha)^{-1/\alpha} r^{-d}t^{(\beta-d)/\alpha} \int_{\RR^d} |u_0(y)|  p(r^{-\alpha},r^{-1}x, (rt^{1/\alpha})^{-1}y) \,dy\,dr \\ 
&= \int_{1-\varepsilon}^1 (1-r^\alpha)^{-1/\alpha} t^{\beta/\alpha} \int_{\RR^d}   p(t,t^{1/\alpha}x,y)  |u_0(y)|\,dy\,dr \le \frac{\alpha \ve^{(\alpha-1)/\alpha}}{\alpha-1}  P_t^*|u_0|(x). 
\end{align*}
Next, 
\begin{align*}
&{I_2} =\int_{1-\ve}^1 (1-r^\alpha)^{-1/\alpha} r^{d-\beta} \int_{\RR^d}   p(1-r^\alpha,x,r w)  H(r^\alpha t,w) \,dw\,dr \\
&= \int_{1-\ve}^1 (1-r^\alpha)^{-1/\alpha} r^{d-\beta} \int_{\RR^d}   p(1-r^\alpha,x,r w)\\
&\ \ \ \times \int_0^{1} (1-s^\alpha)^{-1/\alpha} s^{d-\beta} \int_{\RR^d}   p(1-s^\alpha,w,s y)    |u^*(r^\alpha s^\alpha t,w)|\,dy\,ds\,dw\,dr \\ 
&= \int_{1-\ve}^1\int_0^{1}  (1-r^\alpha)^{-1/\alpha}(1-s^\alpha)^{-1/\alpha} (rs)^{d-\beta}   \int_{\RR^d}   p(1-(rs)^\alpha,x,rs y)    |u^*((rs)^\alpha t,w)|\,dy\,ds\,dr.
\end{align*}
Substituting  $s = v/r$ leads to
\begin{align*}
{I_2} &=  \int_{1-\ve}^1\int_0^{r} \frac{(1-(v/r)^\alpha)^{-1/\alpha}}{r(1-r^\alpha)^{1/\alpha}} v^{d-\beta}   \int_{\RR^d}   p(1-v^\alpha,x,v y)    |u^*(v^\alpha t,w)|\,dy\,dr\,dv.
\end{align*}
By changing the order of integration we get
\begin{align*}
\int_{1-\ve}^1 \int_{0}^{r}\int_{\RR^d} (\ldots)\,dy\,dv\,dr&=  
\int_{0}^{1-\ve} \int_{1-\ve}^{1}\int_{\RR^d} (\ldots)\,dy\,dr\,dv\\
&+\int_{1-\ve}^{1} \int_{v}^{1}\int_{\RR^d} (\ldots)\,dy\,dr\,dv = J_1 + J_2.
\end{align*}
Since {for $v\in(0,1-\ve)$}
\begin{align*}
\int_{1-\ve}^1  \frac{(1-(v/r)^\alpha)^{-1/\alpha}}{r(1-r^\alpha)^{1/\alpha}} dr
&\leq\int_{1-\ve}^1  \frac{dr}{(1-r^\alpha)^{1/\alpha}(r^\alpha-(1-\ve)^\alpha)^{1/\alpha}}=c_\ve\leq c_\ve(1-v^{\alpha})^{-1/\alpha},
%\\
%&\lesssim \frac{1}{(1-\varepsilon)^2}\int_{1-\ve}^{1-\ve/2} \frac{1}{\ve^{1/\alpha}(r-(1-\ve))^{1/\alpha}} dr+ \frac{1}{(1-\varepsilon)^2}\int_{1-\ve/2}^1 \frac{1}{(1-r)^{1/\alpha}\ve^{1/\alpha}} dr\\ 
%&\lesssim \frac{1}{(1-\varepsilon)^2}\int_{1-\ve}^1 \frac{1}{(1-r)^{1/\alpha}(1-v^\alpha)^{1/\alpha}} dr\\ 
%&= \frac{\alpha \ve^{(\alpha-1)/\alpha}}{(\alpha-1)(1-\ve)^2} (1-v^\alpha)^{-1/\alpha},
\end{align*}
we obtain
\begin{align*}
J_1 \le c_\ve H_\ve(t,x),
\end{align*}
for some constant $c_\varepsilon$ depending on $\ve$ and $\alpha$.
In order to estimate $J_2$, let us recall \cite[Lemma 4.3 ]{MR4105374}, which says that  for $v\in (1-\ve,1)$,
\begin{align*}
\int_{v}^1  \frac{(1-(v/r)^\alpha)^{-1/\alpha}}{r(1-r^\alpha)^{1/\alpha}} dr
&= \int_{v}^1 (1-r^\alpha)^{-1/\alpha}(r^\alpha-v^\alpha)^{-1/\alpha} dr \le c\ve^{(\alpha-1)/\alpha}(1-v^\alpha)^{-1/\alpha},
\end{align*}
where $c$ does not depend on $\ve$ and $v$. Hence, $J_2 \le c\ve^{(\alpha-1)/\alpha} h_\ve(t,x)$ and consequently
\begin{align*}
h_\ve(t,x) \le c(\ve) \big(P_t^*|u_0|(x) + H_\ve(t,x) \big)+ c\ve^{(\alpha-1)/\alpha}h_\ve(t,x),
\end{align*}
for some $c,c(\ve)>0$, where $c$ does not depend on $\ve$. Finally, taking sufficiently small $\ve$ we get 
$$h_\ve(t,x) \le \frac{c(\ve)}{1-c\ve^{(\alpha-1)/\alpha}}\big(P_t^*|u_0|(x) + H_\ve(t,x)\big),$$ which yields \eqref{eq0:genest}, what was to be shown.

\end{proof}

\begin{proof}[Proof of Theorem \ref {thm:existence}]

Let $u_{0}^{(k,n)} = u_0 \mathbf{1}_{B(0,k)}\mathbf1_{\{-n \le u_0\le k\}}$.  Then, the functions $u_0^{(k,n)} \in L^1 \cap L^\infty$  and,  by \eqref{eq:BKWuin}, there exists a double sequence of solutions $u^{(k,n)}(t,x)$ to the problems 
\begin{equation*}
\begin{cases}
u_t = \Delta^{\alpha/2} u + b\cdot \nabla\(u^{1+q}\), \qquad t>0, \; x \in \RR^d,\\
u(0,x)=u_0^{(k,n)}(x).
\end{cases}
\end{equation*}
%Now, let us evoke  \cite[Corollary 3.1]{MR1881259} that implies the monotonicity property of solutions. Namely, if $u$ and $v$ are the solutions of (1.1) with initials conditions
%$u_0, v_0\in L^1(\RR^d)\cap L^{\infty}(\RR^d)  $ and $u_0 \geq v_0$, then
%\begin{align}\label{eq:monotonicity}
%u(t, x) \geq v(t, x),\ \ \ \ \ t > 0, x\in \RR^d. 
%\end{align}
By \eqref{eq:genest} and \eqref{ass:A}, each of them satisfies
\begin{align}
\label{aux14}
|u^{({k,}n)}(t,x)| \le C\,t^{-1/\alpha}  \int_{\RR^d}     |u_0( y)| p^{(d-1)}(t, \tilde x, \tilde y) d y \le C\mu  t^{-\beta/\alpha},
\end{align}
where $C>0$ does not depend on {$k$ and} $n$.  Thus, due to the monotonicity property  \eqref{eq:monotonicity}, the sequence $(u^{(k,n)})_{k\ge1}$ is non-decreasing. This ensures existence of the limit $u^{(\infty,n)}(t,x) := \lim_{k\to\infty} u^{(k,n)}(t,x)$. {Similarly, by \eqref{eq:monotonicity} $(u^{(\infty,n)})_{n\ge1}$ is decreasing, hence} there exists the limit $u(t,x) := \lim_{n\to\infty} u^{(\infty,n)}(t,x)$. We will show that this is the solution we are looking for.

In view of \eqref{aux14}, $u(t,x)$ satisfies \eqref{eq:existence}. It remains to prove that $u$ satisfies \eqref{eq:duhamel} as well. Clearly 
$
u(0,x) = 
%\lim_{n \to \infty} u^{(n)}(0,x) = \lim_{n \to \infty} u_0^{(n)}(x) = 
u_0(x).   
$
Since all $u^{(k,n)}$ satisfy \eqref{eq:duhamel}, we have
\begin{align}\label{aux15}
u(t,x) = \lim_{n \to \infty} {\lim_{k \to \infty}} \left(P_t u_0^{(k,n)}(x) + \int_0^t \int_{\RR^d} b \cdot \nabla_x p(t-s,x,z) {\fp{[u^{(k,n)}(s,z)]}{1+q}} dz ds \right).
\end{align}
By \eqref{eq:p<p(d-1)},
\begin{align*}
|p(t,x,y)u_0^{(k,n)}(y)|\leq  t^{-1/\alpha}p^{(d-1)}(t,\tilde x,\tilde y)|u_0(y)|.
\end{align*}
Since the right-hand side is integrable thanks to the condition \eqref{ass:A},  the dominated convergence theorem gives us $\lim_{n \to \infty}\lim_{k \to \infty} P_t u_0^{(k,n)}(x) = P_t u_0(x)$, and therefore we only need to show that we can pass with the limit under the integral in \eqref{aux15}. By Theorem \ref{thm:tildeunorm}, $|b \cdot \nabla_x p(t-s,x,z) u^{(k,n)}(s,z)^{1+q}| \le C|b|(t-s)^{-1/\alpha} s^{-(\beta+\alpha-1)/\alpha} p(t-s,x,z)$ with $C>0$ independent of $k$ and $n$. Thus,  for every $s \in (0,t)$,
\begin{align*}
 \lim_{n \to \infty}\lim_{k \to \infty}  \int_{\RR^d} b \cdot \nabla_x p(t-s,x,z) {\fp{[u^{(k,n)}(s,z)]}{1+q}} dz = \int_{\RR^d} b \cdot \nabla_x p(t-s,x,z) {\fp{u(s,z)}{1+q}} dz.
\end{align*}
Like in the proof of Lemma \ref{lem:aux1} let $\overline u^{(k,n)}$ be the solution to \eqref{eq:problem} with the initial condition  $|u_0^{(k,n)}|$. {Clearly, $|u^{(k,n)}|\le \overline u^{(k,n)}$}. By Theorem \ref{thm:tildeunorm}, $\overline u^{(k,n)}(s,z) \le cs^{-\beta/\alpha}$ with $c$ not depending on $k$ and $n$. Hence, by \eqref{eq:gradp}, \eqref{eq:pscaling}, substituting $z={s^{1/\alpha}}w$ and then using   Lemma \ref{lem:aux1} and the assumption \eqref{ass:A},
%  \eqref{ass:A}, 
  for $s \in (0,t/2)$ we get 
\begin{align*}
&\left|\int_{\RR^d} b \cdot \nabla_x p(t-s,x,z) {\fp{[u^{(k,n)}(s,z)]}{1+q}} dz \right| \\
%&\lesssim  |b| (t-s)^{-1/\alpha} s^{-(\alpha-1)/\alpha} \int_{\RR^d} p(t-s,x,z) \overline u^{(k,n)}(s,z) dz \\
&\lesssim |b| (t-s)^{-1/\alpha} s^{-(\beta+\alpha-1)/\alpha} \int_{\RR^d} p(t-s,{x},{z})  \big(\overline u^{(k,n)}\big)^*(s,s^{-1/\alpha}z) \,dz \\
&= |b| (t-s)^{-1/\alpha} s^{(d+1-\beta-\alpha)/\alpha}t^{-d/\alpha} \int_{\RR^d} p(1-(s/t),t^{-1/\alpha}{x},(s/t)^{1/\alpha}{w})  \big(\overline u^{(k,n)}\big)^*(s,w)\, dw \\
&\lesssim |b| t^{-1/\alpha} s^{(d+1-\beta-\alpha)/\alpha}t^{-d/\alpha}\frac{(s/t)^{(\beta-d)/\alpha}}{(1-(s/t)^{1/\alpha})^{1/\alpha}}t^{(\beta-1)/\alpha}\int_{\RR^d}p^{(d-1)}(t,\tilde x, \tilde y)|u_0(y)|dy\\
&\leq \frac{|b| t^{-(\beta+1)/\alpha} \mu}{(1-1/2^{1/\alpha})^{1/\alpha}} s^{-(\alpha-1)/\alpha}.
\end{align*}
On the other hand, for $s \in [t/2,t)$, by Theorem \ref{thm:tildeunorm},
\begin{align*}
\left|\int_{\RR^d} b \cdot \nabla_x p(t-s,x,z) {\fp{[u^{(k,n)}(s,z)]}{1+q}} dz \right| 
&\lesssim  |b| (t-s)^{-1/\alpha} s^{-(\beta+\alpha-1)/\alpha} \int_{\RR^d} p(t-s,x,z) dz \\
&\lesssim |b|t^{-(\beta+\alpha-1)/\alpha} (t-s)^{-1/\alpha}  .
\end{align*}
Since $\int_0^t (s^{-(\alpha-1)/\alpha} \vee (t-s)^{-1/\alpha}) ds < \infty$, we apply the {dominated} convergence theorem once again and get
\begin{align*}
 &\lim_{n \to \infty}\lim_{k \to \infty}  \int_0^t \int_{\RR^d} b \cdot \nabla_x p(t-s,x,z) {\fp{[u^{(k,n)}(s,z)]}{1+q}} dz ds \\
&= \int_0^t\int_{\RR^d} b \cdot \nabla_x p(t-s,x,z) {\fp{u(s,z)}{1+q}} dz ds,
\end{align*}
which ends the proof.
\end{proof}

\section{The self-similar solution}
This section is devoted to the proof of Theorem \ref{thm:sss}, where the initial condition $u_0(x)=M|x|^{-\beta}$, $1<\beta<d$ and $M>0$ is considered.

\subsection{Existence and selfsimilarity}

Let us start with the following  observation on how scaling of the initial condition is transferred to scaling of the solution. 
\begin{lemma}\label{prop:rescaled} If $u$ is a solution to \eqref{eq:problem} with initial condition $u_0$,
%satisfying the condition \eqref{ass:A}
 then $v(t,x)=\lambda^{\beta/\alpha}u(\lambda t,\lambda^{1/\alpha }x)$ is a solution with the initial condition $v_0(x)=\lambda^{\beta/\alpha}u_0(\lambda^{1/\alpha }x).$
\end{lemma}
\begin{proof}
By scaling property \eqref{eq:pscaling}, we have
\begin{align*}
\lambda^{\beta/\alpha}P_{\lambda t} u_0(\lambda^{1/\alpha} x) &= \lambda^{\beta/\alpha} \int_{\RR^d} p(\lambda t, \lambda^{1/\alpha} x,y) u_0(y) \,dy\\
&= \int_{\RR^d} \lambda^{d/\alpha}p(\lambda t, \lambda^{1/\alpha} x,\lambda^{1/\alpha}w) \lambda^{\beta/\alpha}u_0(\lambda^{1/\alpha} w) \,dw = P_t v_0(x)\,.
\end{align*}
Similarly,
\begin{align*}
&\lambda^{\beta/\alpha}\int_0^{\lambda t} \int_{\RR^d} b\cdot \nabla_x p(\lambda t- s, \lambda^{1/\alpha} x,y) u^{q+1}(s,y) \,dy\,ds \\
&= \int_0^{t} \int_{\RR^d} \lambda^{d/\alpha +1}b\cdot \nabla_x p(\lambda t - \lambda r, \lambda^{1/\alpha} x,\lambda^{1/\alpha} w) \lambda^{(1-\alpha)/\alpha}\lambda^{\beta(q+1)/\alpha} u^{q+1}(\lambda r,\lambda^{1/\alpha} w) \,dw\,dr \\
&= \int_0^{t} \int_{\RR^d} b\cdot \nabla_x p(t-r, x,w) v(r,w) \,dw\,dr.
\end{align*}
Hence,
\begin{align*}
v(t,x) = \lambda^{\beta/\alpha} u(\lambda t,\lambda^{1/\alpha} x) = P_t v_0(x) + \int_0^{t} \int_{\RR^d} b\cdot \nabla_x p(t-r, x,w) v(r,w) \,dw\,dr,
\end{align*}
%\begin{align*}
%v_t(t,x)&=\lambda^{1+\beta/\alpha}u_t(\lambda t,\lambda^{1/\alpha }x)\\
%&=\lambda^{1+\beta/\alpha}\[\(\Delta^{\alpha/2} u\)(\lambda t, \lambda^{1/\alpha }x) + b\cdot \(\nabla u\)(\lambda t, \lambda^{1/\alpha}x)u^{(\alpha-1)/\beta}(\lambda t, \lambda^{1/\alpha }x)\]\\
%&=\(\lambda\Delta^{\alpha/2} (\lambda^{\beta/\alpha}u)\)(\lambda t, \lambda^{1/\alpha }x) + b\cdot \(\lambda^{1/\alpha}\nabla (\lambda^{\beta/\alpha}u)\)(\lambda t, \lambda^{1/\alpha}x)(\lambda^{\beta/\alpha}u)^{(\alpha-1)/\beta}(\lambda t, \lambda^{1/\alpha }x)\\
%&=\Delta^{\alpha/2} v + (b\cdot \nabla v)v^{(\alpha-1)/\beta}.
%\end{align*}
as required.
\end{proof}
Next we verify the condition \eqref{ass:A} in the case $u_0(x)=M|x|^{-\beta}$.
\begin{lemma}\label{lem:u_0ass}
The function $u_0(x)=M|x|^{-\beta}$, $M>0$, satisfies the condition \eqref{ass:A}.
\end{lemma}
\begin{proof}
%By scaling property \eqref{eq:pscaling} and the equivalence \eqref{eq:intp}, we have
%\begin{align*}
%\int_\RR |y|^{-\beta} dy_1 &= c \int_\RR \int_0^\infty p(t,y) t^{(d-\beta-\alpha)/\alpha} dt dy_1
% = c \int_0^\infty p^{(d-1)}(t,\tilde{y}) t^{(d-\beta-\alpha)/\alpha} dt = c_1|\tilde{y}|^{1-\beta}.
%\end{align*}
%Hence, by \ref{lem:estpa} 
Since
\begin{align}\label{eq:inty^beta}
\int_\RR |y|^{-\beta} dy_1 \leq \int_\RR [\tfrac12(|\tilde y|+|y_1|)]^{-\beta} dy_1 =\frac{2^{\beta+1}}{\beta-1}|\tilde y|^{1-\beta},
\end{align} 
Lemma \ref{lem:estpa} gives us
\begin{align}\label{eq:Uest}
\int_{\RR^{d}} u_0(y) p^{(d-1)}(t, \tilde x, \tilde y)  d\tilde{y} \lesssim   \int_{\RR^{d-1}} |\tilde{y}|^{1-\beta} p^{(d-1)}(t, \tilde x, \tilde y)  d\tilde{y} \lesssim (t^{1/\alpha}+|\tilde{x}|)^{1-\beta},
\end{align}
which yields \eqref{ass:A}.
\end{proof}
Therefore, by Theorem \ref{thm:existence} there exists a solution $u(t,x)$ to the initial problem \eqref{eq:problem} with $u_0(x) = M|x|^{-\beta}$, $M>0$.
From Lemma \ref{prop:rescaled} we know that its rescaled version $\lambda^{\beta}u(\lambda^\alpha t,\lambda x)$, $\lambda>0$, is also a solution to the same problem. Nevertheless, these may be two different solutions. In  Lemma  \ref{lem:selfsimilar} we show that the solution from Theorem \ref{thm:existence} is self-similar i.e. coincides with its rescaled versions. 

Throughout the whole section we exploit an analogous  notation to the one from the proof of Theorem \ref{thm:existence}. Namely,  $u^{(n)}$, $n\geq1$, is defined as  the solution to the problem \eqref{eq:problem} with

\begin{align}\label{eq:approxu0}
	u_{0}^{(n)} (x)= \(n\wedge\frac M{|x|^\beta}\)\mathbf{1}_{B(0,n)}(x),
\end{align}
   and $u(t,x):=\lim_{n\nearrow\infty}u^{(n)}(t,x)$, $x\in\RR^d$. The initial condition is positive in this case, so there is no need to consider the double sequence $u_0^{(k,n)}$.

\begin{lemma}\label{lem:selfsimilar}
For any $\lambda>0$ we have
 \begin{align}\label{eq:selsimilar}
	 u(t,x)=\lambda^{\beta}u(\lambda^\alpha t,\lambda x),\ \ \ \ \ \ t>0, x\in\RR^d.
 \end{align}
\end{lemma}
\begin{proof}
It suffices to consider $\lambda>1$. Indeed, if we substitute $t=s/\lambda^\alpha$ and $x=y/\lambda$ in \eqref{eq:selsimilar}, we get the result for $\lambda<1$. We note that
$$\lambda^{\beta}u(\lambda^\alpha t,\lambda x)=\lim_{n\to\infty}\lambda^{\beta}u^{(n)}(\lambda^\alpha t,\lambda x),$$
where, by Lemma \ref{prop:rescaled}, $\lambda^\beta u^{(n)}(\lambda^\alpha t,\lambda x)$ is the solution to the problem \eqref{eq:problem} with the rescaled initial condition $\lambda^{\beta}u^{(n)}_0(\lambda x)=\(\lambda ^{\beta}n\wedge\frac M{|x|^\beta}\)\mathbf{1}_{B(0,n/\lambda)}$.
Observe that for $i_n=\lfloor n/\lambda\rfloor$ and $j_n=\lceil \lambda^\beta n\rceil$, 
$$u^{(i_n)}_0(x)\leq \lambda^\beta u^{(n)}_0(\lambda x)\leq u^{(j_n)}_0(x),\ \ \ \ \ x\in\RR^d.$$
Here  $\lfloor \cdot\rfloor$ and $\lceil \cdot \rceil$ are the standard floor and ceiling functions, respectively. Hence, by the   monotonicity property \eqref{eq:monotonicity}, we get 
$$u^{(i_n)}(t,x)\leq \lambda^{\beta}u^{(n)}(\lambda^\alpha t,\lambda x)\leq u^{(j_n)}(t,x), \qquad t>0,\, x\in \RR^d,$$
which  yields $\lim_{n\to\infty}\lambda^{\beta}u^{(n)}(\lambda^\alpha t,\lambda x)=u(t,x)$, as required.
\end{proof}
In order to study the self-similar solution  from Lemma \ref{lem:selfsimilar}, it is clearly enough to study the case $t=1$. Namely, denoting $U(x):=u(1,x)$ we get
\begin{align}\label{eq:u=U}
u(t,x)=t^{-\beta/\alpha}U(t^{-1/\alpha}x).
\end{align}
In particular, this implies $u^*(t,x)=U(x)$, and the Duhamel formula \eqref{eq:duhrescalled} takes  the form
\begin{equation}\label{eq:duh1}
U(x) = M\int_{\RR^d}\frac{p(1,x,y)}{|y|^\beta}dy+\alpha \int_0^{1} \int_{\RR^d} r^{d-\beta} \nabla_xp(1-r^\alpha,x,r w) U(w)^{q+1}\,dw\,dr. 
\end{equation}

\subsection{Pointwise estimates and asymptotics}
From Proposition \ref{thm:genest} and the estimate \eqref{eq:Uest} we have the following upper bound 
\begin{align}\label{eq:u1est1}
U(x)\lesssim \frac 1{(1+|\tilde x|)^{\beta-1}}.
\end{align} 
However, it is not optimal. In particular, the right-hand side does not depend on $x_1$. In the sequel, we provide precise  two-sided estimates. We start with  $L^p$ bounds.

\begin{lemma}\label{prop:Lpest}
For any $\gamma\in(d/\beta,\infty]$, $U \in L^\gamma(\RR^d)$.
\end{lemma}
\begin{proof} {For $\gamma=\infty$ the result follows by \eqref{eq:u1est1}}. Furthermore, by monotonicity of $u^{(n)}$, \eqref{eq:u1est1},  Corollary \ref{lem:mcu0} and Lemma \ref{lem:estpa}, for $\gamma\in\big((d-1)/(\beta-1),\infty\big)$ we have
\begin{align}\nonumber
\int_{\RR^d}U^\gamma(x)dx&=\lim_{n\rightarrow\infty}\int_{\RR^d}[u^{(n)}(1,x)]^\gamma dx\\\nonumber
&\lesssim \lim_{n\rightarrow\infty}\int_{\RR^d}\frac 1{(1+|\tilde x|)^{(\beta-1)(\gamma-1)}}u^{(n)}(1,x)\,dx\\\nonumber
&= \lim_{n\rightarrow\infty}\int_{\RR^d}\frac 1{(1+|\tilde x|)^{(\beta-1)(\gamma-1)}}P_1(u_0^{(n)})(x)\,dx\\\nonumber
&\lesssim \int_{\RR^d}\frac 1{(1+|\tilde x|)^{(\beta-1)(\gamma-1)}}\frac{1}{(1+|\tilde x|+|x_1|)^{\beta}}\,dx\\\label{eq0:Lpest2}
&=\frac{2}{\beta-1} \int_{\RR^{d-1}}\frac 1{(1+|\tilde x|)^{(\beta-1)\gamma}}\,d \tilde{x}<\infty.
\end{align}
% By Proposition \ref{thm:genest}, we have \gs{TEN ARGUMENT JEST NIEPRAWDZIWY. TRZEBA TO POPRAWIC}
%\begin{align*}
%\|U\|_{d/(\beta-1)}&\lesssim\( \int_{\RR^{d-1}} \frac 1{(1+|\tilde x|)^{d}}d\tilde x\)^{{(\beta-1)/d}} <\infty.
%\end{align*}
%Since $\|U\|_{\infty} < \infty$ (see \eqref{eq:u1est1}), by interpolation we obtain that
%\begin{align}\label{eq0:Lpest}
%\|U\|_\gamma < \infty, \qquad \gamma\geq d/(\beta-1).
%\end{align}
Now, let $\gamma\in(\frac{d}{\beta},\frac{d-1}{\beta-1})$. By the integral Minkowski inequality, we have
\begin{align}\nonumber
\|U\|_{\gamma}&\lesssim \|P_1u_0\|_{\gamma} \\\nonumber
&\ \ \ +\(\int_{\RR^d}\(\int_0^{1} \int_{\RR^d} \frac{r^{d-\beta}}{ (1-r)^{1/\alpha}}p(1-r,rw)[ U(w+\frac{x}{r})]^{1+q}\,dw\,dr\)^{\gamma}dx\)^{1/\gamma}\\\nonumber
&\leq  c_{\gamma}+\|U^{1+q}\|_{\gamma}\int_0^{1} \int_{\RR^d} \frac{r^{d-\beta+d/\gamma}}{ (1-r)^{1/\alpha}}p(1-r,rw)\,dw\,dr\\\label{aux3}
&=  c_{\gamma}+\|U\|_{\gamma(1+q)}^{1+q}\int_0^{1}  r^{-\beta+d/\gamma} (1-r)^{-1/\alpha}\,dr \le c_0\left(1+\|U\|_{\gamma(1+q)}^{1+q}\right).
\end{align}
%where we bounded $-\beta+d/p\geq -\beta +(\beta-1)(1+q)\geq \alpha-2-\frac{\alpha-1}\beta>-1$. 
Thus, for every $m \in \NN$ such that $\gamma(1+q)^m > (d-1)/(\beta-1)$, by iterating \eqref{aux3} we get 
\begin{align}\label{eq:Lpest}
\|U\|_{\gamma}&\le  c_m\left(1+\|U\|_{\gamma(1+q)^{m+1}}^{(1+q)^{m+1}}\right). 
\end{align}
We take the smallest $m\in\NN$ such that  that $\gamma(1+q)^{m+1} > (d-1)/(\beta-1)$. Then, by \eqref{eq:Lpest} and \eqref{eq0:Lpest2}, $U \in L^\gamma(\RR^d)$.

\end{proof}
The next step is to show that $U$ vanishes at infinity.

\begin{lemma}\label{lem:U->0}
We have
 $$\lim_{R\rightarrow\infty}\sup_{|x|>R}U(x)=0.$$
\end{lemma}
\begin{proof} 
Estimates \eqref{eq:pest} and  Lemma \ref{lem:estpa} applied to \eqref{eq:duh1} give us
\begin{align}\label{eq:DuhamelEst}
U(x) &\lesssim \frac{M}{(1+|x|)^\beta}+\int_0^{1} \int_{\RR^d} {r^{d-\beta}}\frac{p(1-r,x,r w)}{{ (1-r)^{1/\alpha}}+|x-rw|} U(w)^{q+1}\,dw\,dr.
\end{align} 
Let us split the integral above into
$I_1+I_2:=\int_0^{\varepsilon}(\ldots)dr +\int_{\varepsilon}^1(\ldots)dr,$
for some $\varepsilon\in(0,1/2)$. First, note that
\begin{align*}
\int_{\RR} |y|^{-\beta} dy_1 = \int_{\RR} (y_1^2 + |\tilde{y}|^2)^{-\beta/2} dy = c |\tilde{y}|^{1-\beta}.
\end{align*}
Hence, by \eqref{aux1}, \eqref{eq:inty^beta} and Lemma \ref{lem:estpa} we have
\begin{align*}
I_1 &\lesssim \|U\|_\infty^{q}\int_0^\varepsilon (1-r)^{-2/\alpha} \int_{\RR^{d}} p^{(d-1)}(1,\tilde x,\tilde y) |y|^{-\beta} d y \\
&\lesssim \ve\|U\|_\infty^{q} \int_{\RR^{d-1}} p^{(d-1)}(1,\tilde x,\tilde y) |\tilde{y}|^{1-\beta} d \tilde{y} \lesssim \varepsilon \|U\|_\infty^{q}. 
\end{align*}
Next, using H\"older inequality and the bound \eqref{eq:Lpp} we get
\begin{align*}
I_2&\lesssim\int_\varepsilon^1{r^{d-\beta}}\int_{B(0,|x|/2)} \frac{p(1-r,x,r w)}{|x|} \|U\|_\infty^{q+1}\,dw\,dr\\
&\ \ \ +\int_\varepsilon^1\frac{{r^{d-\beta}}}{(1-r)^{1/\alpha}}\int_{B(0,|x|/2)^c} p(1-r,x,r w) U(w)\|U\|_\infty^{q}\,dw\,dr\\
&\leq\frac{\|U\|_\infty^{q+1}}{|x|}\int_\varepsilon^1{r^{-\beta}}dr\\
&\ \ \ +\|U\|_\infty^{q}\int_\varepsilon^1\frac{{r^{d-\beta}}}{(1-r)^{1/\alpha}} \|p(1-r,x,r (\cdot))\|_{2d/(2d-\alpha+1)} \|\mathbf1_{\{B(0,|x|/2)^c\}}U\|_{2d/(\alpha-1)}\,dr\\
&\lesssim\frac{\|U\|_\infty^{q+1}}{\varepsilon^\beta |x|}+\|U\|_\infty^{q}\|\mathbf1_{\{B(0,|x|/2)^c\}}U\|_{2d/(\alpha-1)}\int_\varepsilon^1\frac{{r^{(\alpha-1)/2-\beta}}}{(1-r)^{1-(\alpha-1)/2\alpha}}  \,dr,
\end{align*}
which tends to zero as $|x|\rightarrow\infty$ by virtue of Lemma \ref{prop:Lpest} applied with $\gamma=2d/(\alpha-1)>d/\beta$. Thus, the integral in \eqref{eq:DuhamelEst} is arbitrarly small for large $|x|$, which ends the proof.
\end{proof}

We are now ready to derive the upper bound of $U$.

\begin{prop}\label{thm:Uest}
There is a constant $C=C(d,\alpha,b,M,\beta)$ such that
\begin{align*}
 U(x) &\le C \frac{1}{(1+|x|)^\beta},\ \ \ \ \ x\in \RR^d.
 \end{align*}
\end{prop}

\begin{proof}  Recall that $u^{(n)}$ is a solution to the problem \eqref{eq:problem} with initial condition given by \eqref{eq:approxu0} and $\(u^{(n)}\)^*(t,x)$ is defined by \eqref{def:star}. Note also that $\(u^{(n)}\)^*(t,x) \nearrow u^*(t,x) = U(x)$. From \eqref{eq:estJDE} and Lemma \ref{lem:estpa} we have
\begin{align*}
A_n:=\sup_{\substack{t>0\\x\in\RR^d}}{\(u^{(n)}\)^*(t,x)}(1+|x|)^\beta<\infty.
\end{align*}
It is enough to show that $A_n$'s are uniformly bounded.
Duhamel formula  \eqref{eq:duhrescalled} combined with  \eqref{eq:gradp} give us
\begin{align*}
\(u^{(n)}\)^*(t,x) \lesssim  &\,C_0\left(\frac1{(1+|x|)^\beta}  +I_n(t,x)\right),
\end{align*}
where $C_0=C_0(d,\alpha,b,M)>0$ is some constant and
\begin{align*}
 I_n(t,x)= \int_0^{1} \int_{\RR^d} r^{d-\beta} \frac{p(1-r^\alpha,x,r w)}{(1-r^\alpha)^{1/\alpha}+|x-rw|}\[ \(u^{(n)}\)^*(r^\alpha t,w)\]^{(\alpha-1+\beta)/\beta}\,dw\,dr.
\end{align*}
Let $0<\ve<1/2$. By Lemma \ref{lem:U->0}, we may and do choose  $R>0$  such that $|U(z)|<\varepsilon^{\beta/(\alpha-1)}$ for $|z|\geq R$ and  such that ${\|U\|^{(\alpha-1)/\beta}_\infty}/{|x|}<\varepsilon$ for $|x|>2R$.

Now let $|x|>2R$. Since $|x-rw|\geq |x|/2$ for $|w|<R$ and $0<r<1$,  we get
\begin{align*}
I_n(t,x)&\le \int_0^{1} \int_{B(0,R)}r^{d-\beta}\frac{p(1-r^\alpha,x,rw)}{|x|/2}\(u^{(n)}\)^*(r^\alpha t,w)\|U\|_\infty^{(\alpha-1)/\beta}dw\,dr \\
&\ \ \ +\int_0^{1} \int_{B(0,R)^c}r^{d-\beta}\frac{p(1-r^\alpha,x,rw)}{(1-r^\alpha)^{1/\alpha}}\(u^{(n)}\)^*(r^\alpha t,w)\(\sup_{|z|\geq R}U(z)\)^{(\alpha-1)/\beta}dw\,dr\\
& \le \varepsilon A_n\(2\|U\|_\infty^{(\alpha-1)/\beta}+1\)\int_0^{1} \int_{\RR^d} r^{d-\beta}(1-r^\alpha)^{-1/\alpha} \frac{p(1-r^\alpha,x,r w)}{(1+|w|)^\beta}\,dw\,dr.
\end{align*}
By \eqref{eq:pscaling} and Lemma \ref{lem:estpa},
\begin{align}\label{aux5}
&\int_{\RR^d} \frac{p(1-r^\alpha,x,r w)}{(1+|w|)^\beta}\,dw=\int_{\RR^d} r^{-d} \frac{p(r^{-\alpha}-1,r^{-1}x, w)}{(1+|w|)^\beta}\,dw\approx\frac{ r^{\beta-d}}{(1+|x|)^\beta}.
\end{align}
Hence,
\begin{align}\nonumber
I_n(t,x)& \lesssim \frac{\varepsilon A_n}{(1+|x|)^\beta}\int_0^{1}  (1-r^\alpha)^{-1/\alpha}dr\approx \frac{\varepsilon A_n}{(1+|x|)^{\beta}}.
\end{align}
On the other hand, for $|x|\leq 2R$, 
\begin{align*}
\(u^{(n)}\)^*(t,x)\leq U(x)\leq \|U\|_\infty\leq \frac{\|U\|_\infty(1+2R)^\beta}{(1+|x|)^\beta}.
\end{align*}
Finally, there is a constant $c=c(d,\alpha,b,M,\beta)$ such that for any $t>0$ and $x\in\RR^d$ we have
$$\(u^{(n)}\)^*(t,x)\le \frac{c}{(1+|x|)^\beta}\(1+\varepsilon A_n+\|U\|_\infty(1+2R)^\beta\),$$
and consequently 
$$A_n\le c(1+\varepsilon A_n+\|U\|_\infty(1+2R)^\beta).$$
Taking $\varepsilon=1/(2c)$, we obtain 
$$A_n\le {2(1+\|U\|_\infty(1+2R)^\beta)},$$
as required. 

\end{proof}

Next we provide an asymptotic behaviour of $U$, which is driven by the functions $h_n$ below. First, let us define  the kernel
$$k(x,w)=\alpha \int_0^{1} r^{d-\beta} \nabla_x p(1-r^\alpha,x,r w)dr,$$
and then, for a function $f:\RR^d\rightarrow\RR$, the operator 
\begin{align*}
(K f)(x) &= P_1u_0(x)+\int_{\RR^d}k(x,w)f(w)| f(w)|^{q}\,dw. 
\end{align*}
In view of \eqref{eq:duh1}, we clearly have $KU=U$ . Furthermore, we put  $h_0\equiv0$ and 
\begin{align*} h_n(x) = {K}^nh_0(x),\qquad x\in\RR^d, n\ge1. 
\end{align*} 
In particular,
\begin{align*}
h_1(x) &= P_1u_0(x), \\
h_2(x) &= P_1u_0(x) + \int_{\RR^d} k(x,w) [P_1 u_0(w)]^{q+1}dw.
\end{align*}
One can  show inductively that for any $n\geq0$ there is a constant $c_n>0$ such that
\begin{align}\label{eq:h_nest} |h_n(x)|\leq \frac{c_n}{(1+|x|)^{\beta}},\ \ \ \ \ x\in\RR^d.\end{align}
Indeed, by  \eqref{aux5} we get
\begin{align*}
|h_{n+1}(x)|&=|Kh_n(x)|\\
&\lesssim P_1u_0(x)+ \int_0^{1} \int_{\RR^d} r^{d-\beta} |\nabla_x p(1-r^\alpha,x,r w)|| h_n(w)|^{q+1}\,dw\,dr\\[5pt]
&\lesssim  \frac{1}{(1+|x|)^\beta}+c_n^{q+1}\int_0^1\frac{r^{d-\beta}}{(1-r)^{1/\alpha}}\int_{\RR^d}\frac{p(1-r,x,rw)}{(1+|w|)^{\beta}}dw\, dr\\[5pt]
&\lesssim\frac{c_n^{q+1}+1}{(1+|x|)^\beta} .
\end{align*}
Note that the functions $h_n$ might be considered as Picard approximations of $U$. 
The next result supplements Proposition \ref{thm:Uest} precisely describing behaviour of $U(x)$ for large $|x|$. Although asymptotics of  the first two orders only are mentioned in Theorem \ref{thm:sss}, we provide below more general result, as the arguments are similar.

\begin{prop}\label{thm:asymp}  For every $n\geq0$ we have
$$|U(x)-h_n(x)|\stackrel{n}{\lesssim} {(1+|x|)^{-[\beta+n(\alpha-1)]\wedge[d+\alpha+1]}}.$$
\end{prop}
\begin{proof}
We will use induction. For $n=0$ the assertion follows from Theorem \ref{thm:Uest}. Consider some $n\geq0$. Since $U$ and every $h_n$ are bounded, we may focus only on large $|x|$.

Recall that for any $a \in \RR$ and $\gamma>1$ we denote $\fp{a}{\gamma} = a |a|^{\gamma-1}$. Note that $\frac{d}{dx} \fp{x}{\gamma} = \gamma|x|^{\gamma-1}$. Hence, for any $a,b \in \RR$ and $\gamma>0$ we have
\begin{align}\nonumber
\left|\fp{b}{\gamma+1} - \fp{a}{\gamma+1}\right| &= \left|\int_a^b \frac{d}{dx}\fp{x}{\gamma+1} dx \right| = (\gamma+1)\left|\int_a^b |x|^\gamma dx\right|\\\label{eq:french}
& \le (\gamma+1)|b-a|(|a|+|b|)^\gamma.
\end{align}
%\del{The simple inequality  $a^{1+\gamma}-b^{1+\gamma}\le(1+\gamma)(a+b)^\gamma(a-b)$ valid for $a,b,\gamma>0$ may be generalised to
%$$|a|a|^\gamma-b|b|^\gamma|\leq (1+\gamma)|a-b|(|a|+|b|)^\gamma,$$
%where $a$ and $b$ are any real numbers.
%}
Thus,   by  Theorem \ref{thm:Uest}, \eqref{eq:h_nest} and \eqref{eq:french}
\begin{align*}
&|U(x)-h_{n+1}(x)|\\
&\lesssim\int_{\RR^d} |k(x,w)|\left|U(w)-h_n(w)\right|\(|U(w)|+|h_{{n}}(w)|\)^{q}\,dw\\
&\stackrel{n}{\lesssim} C_n\int_{\RR^d}  |k(x,w)|\frac{1}{(1+|w|)^{[\beta+n(\alpha-1)]\wedge[d+\alpha+1]}} \frac{1}{(1+|w|)^{\beta q}}dw\,dr\\
&\lesssim \int_0^{1} r^{d-\beta}\int_{\RR^d}  \frac{p(1-r^\alpha,x,r w)}{(1-r^\alpha)^{1/\alpha}+|x-rw|} \; \frac{1}{(1+|w|)^{[\beta+(n+1)(\alpha-1)]\wedge[d+2\alpha}]} dw\,dr.
\end{align*} 
Let $a_n= \beta+(n+1)(\alpha-1)$. In the case ${a_n}<d$, we bound $|x-rw|$ by zero and take advantage of \eqref{aux5}. Assume therefore ${a_n}\geq d$. Since  $|x-rw| \ge |x|/2$ for $w \in B(0,|x|/2r)$ and $0<r<1$, we split the inner integral above and estimate (applying \eqref{eq:pest}) it as follows 
\begin{align*}
\int_{\RR^d} (\ldots) dw
&\stackrel{n}{\lesssim} \int_{B(0,|x|/2r)} \frac{1-r^\alpha}{((1-r^\alpha)^{1/\alpha}+|x|)^{d+\alpha+1}(1+|w|)^{{a_n\wedge[d+2\alpha]}}} dw\\
&\ \ \ \ +\int_{B(0,|x|/2r)^c} \frac{p(1-r^\alpha,x,r w)}{(1-r^\alpha)^{1/\alpha}(1+|x|/r)^{{a_n\wedge[d+2\alpha]}}} dw\\
&\stackrel{n}{\lesssim} \frac{1}{|x|^{d+\alpha+1}}\int_{B(0,|x|/2{r})}  \frac{dw}{(1+|w|)^{{a_n \wedge[d+2\alpha]}}}+ \frac{r^{-d+\({a_n}\wedge[d+2\alpha]\)}}{(1-r^\alpha)^{1/\alpha}|x|^{{a_n\wedge[d+2\alpha]}}} \\
&\stackrel{n}{\lesssim} \frac{1+\mathbf1_{{a_n}=d}\ln|x/2r|}{|x|^{d+\alpha+1}}+ \frac{{1}}{(1-r^\alpha)^{1/\alpha}|x|^{{a_n\wedge[d+\alpha+1]}}} \\
&{\lesssim}  \frac{1+ {\ln(1/r)}}{(1-r^\alpha)^{1/\alpha}|x|^{{a_n}\wedge[d+\alpha+1]}}.
\end{align*}
Thus, we arrive at 
\begin{align*}
|U(x)-h_{n+1}(x)|&\stackrel{n}{\lesssim}  \frac1{|x|^{[\beta+(n+1)(\alpha-1)]\wedge[d+\alpha+1]}}\int_0^{1} \frac{r^{d-\beta}{(1+\ln(1/r))}}{(1-r^\alpha)^{1/\alpha}}dr\\
&\stackrel{n}{\approx}  \frac1{|x|^{[\beta+(n+1)(\alpha-1)]\wedge[d+\alpha+1]}},
\end{align*}
which ends the proof.
\end{proof}
Proposition \ref{thm:asymp} directly imples the lower bound for $U(x)$.

\begin{cor}\label{cor:Ulowest}
We have $U(x) \gtrsim \dfrac{1}{(1+|x|)^\beta}$ for $x\in \RR^d$. 
\end{cor}
\begin{proof}
  Proposition \ref{thm:asymp} with $n=1$ and  Lemma \ref{lem:estpa} applied to Proposition \ref{thm:Uest} imply  the lower bound for large $|x|$. For any fixed $R>0$ and $|x|<R$ we employ \eqref{eq:estJDE} and bound
\begin{align*}
U(x)&\geq u^{(1)}(1,x)\gtrsim (M\wedge 1)(P_1\mathbf1_{B(0,1)})(x)\approx \int_{|w|<1}\frac{dw}{(1+|w-x|)^{d+\alpha}}\\
&\gtrsim \frac1{(2+R)^{d+\alpha}}\geq\frac{1}{(2+R)^{d+\alpha}}\frac 1{(1+|x|)^{\beta}},
\end{align*}
which proves the assertion.
\end{proof}
\subsection{Gradient of $U$}
Eventually, we turn our attention to  the gradient of $U$.  Note that the  estimates in \eqref{eq:nablaU}, that we are going to prove, coincides with the ones of $U$ from Proposition \ref{thm:Uest}. Nevertheless, it is not true for the whole range of arguments of the solution $u(t,x)$, as the scaling property of the gradient $\nabla u(t,x)$ is slightly  different, i.e. we have
\begin{align*}
(\nabla_x u)(t,x)=\nabla_x \(\lambda^{\beta}u(\lambda^\alpha t,\lambda x)\)= \lambda^{\beta+1}(\nabla_x u)(\lambda^\alpha t,\lambda x),
\end{align*}
 and consequently, for $\lambda=t^{-1/\alpha}$, we have (cf. \eqref{eq:u=U})
 \begin{align*}
 (\nabla_x u)(t,x)=t^{-(\beta+1)/\alpha}(\nabla_x u)(1,x/t^{1/\alpha})=t^{-(\beta+1)/\alpha} (\nabla U)(x/t^{1/\alpha}).
\end{align*}
%In order to estimate $\nabla U$ we will need to rewrite the Duhamel formula in another form. Namely, by using the integration by parts, we get from \eqref{eq:duhamel} \tj{POWINNISMY TO SFORMUŁOWAĆ DLA $u^{(n)}$ (NIE WIEMY CZY MOŻEMY RÓŻNICZKOWAĆ $u$). NIGDZIE NIE UŻYWAMY \eqref{eq:duhamel2}, KASUJEMY TO?}
%\begin{align}\label{eq:duhamel2}
%u(t,x) &= P_tu_0(x) + \tj{(q+1)}\int_0^t \int_{\RR^d} p(t-s,x,z) b \cdot \nabla_z u(s,z)|u(s,z)|^{q}\,dz\,ds.
%\end{align}

\begin{prop}\label{prop:gradUest}
We have $ U\in \mathcal C^1\(\RR^d\)$ and 
\begin{align}\label{eq:nablaU}|\nabla U(x)|\lesssim \frac1{(1+|x|)^\beta}.\end{align}
\end{prop}
\begin{proof}
By \cite[Theorem 3.5]{MR2407204}, since 
$|\nabla u_0^{(n)}(x)|\leq\frac{M\beta}{|x|^{\beta+1}}\mathbf1_{\{1/n<|x|<n\}}$, 
we have 
$|\nabla u^{(n)}(t,x)|\leq c(1+t) (1+|x|)^{-d-\alpha-1}$,
where $c$ may depend on $n$. Hence, for every $n$ we have 
$$A_n:=\sup_{\substack{x\in\RR^d\\0<t\leq1}}\left|t^{1/\alpha}(t^{1/\alpha}+|x|)^{\beta}\nabla u^{(n)}(t,x)\right|<\infty.$$
%\tj{Wygodniej jest tutaj wstawić $P_t u_0(x)$ zamiast $(t^{1/\alpha}+|x|)^{\beta}$}. \gs{Dla mnie nie ma potrzeby}
We will show that 
\begin{align}\label{eq:A<oo}
A:=\sup_{n\in\NN}A_n< \infty.
\end{align} 
From  \eqref{eq:duhamel} {and by applying integration by parts,} for any $\varepsilon\in(0,1/2)$ we get 
\begin{align}\nonumber
\nabla u^{(n)}(t,x) &= \nabla P_tu^{(n)}_0(x) + \int_0^{(1-\varepsilon)t} \int_{\RR^d} \nabla \(b\cdot \nabla p(t-s,x,z) \) [u^{(n)}(s,z)]^{q+1}\,dz\,ds\\\label{aux4}
&\ \ \ + (q+1)\int_{(1-\varepsilon)t}^t \int_{\RR^d} b \cdot \nabla_z p(t-s,x,z)  u^{(n)}(s,z)^{q}\nabla u^{(n)}(s,z)\,dz\,ds.
\end{align}
Let us note that the first integrand is not absolutely integrable on $\big((1-\varepsilon)t,t\big)\times\RR^d$, which explains the above decomposition. 
We estimate the term $\nabla P_tu^{(n)}_0(x)$ as follows
\begin{align*}
|\nabla_x P_tu^{(n)}_0(x)|&=\left|\nabla_x\int_{\RR^d}p(t,x,y)u^{(n)}_0(y)dy\right|=\left|\int_{\RR^d}\nabla_x p(t,x,y)u^{(n)}_0(y)dy\right|\\
&\lesssim t^{-1/\alpha}\int_{\RR^d}p(t,x,y) |y|^{-\beta}dy\lesssim \frac{1}{t^{1/\alpha}\(t^{1/\alpha}+|x|\)^\beta},
\end{align*}
where we used \eqref{eq:gradp} and Lemma \ref{lem:estpa}.
Next, applying this,  \eqref{eq:gradp}, \eqref{eq:u=U},  and  Proposition \ref{thm:Uest}    to \eqref{aux4}, we obtain for $0<t\leq1$ 
\begin{align*}
|\nabla u^{(n)}(t,x)| &\lesssim \frac{1}{t^{1/\alpha}(t^{1/\alpha}+|x|)^{\beta}}\\
&\ \ \  + \int_0^{(1-\varepsilon)t} \frac{1}{(t-s)^{2/\alpha}}\(\frac{\|U\|_\infty}{s^{\beta/\alpha}}\)^{q}P_{t-s}((s^{1/\alpha}+|\cdot|)^{-\beta})(x)ds\\
&\ \ \ + A_n\int_{(1-\varepsilon)t}^t \frac1{(t-s)^{1/\alpha}}\(\frac{\|U\|_\infty}{s^{\beta/\alpha}}\)^{q}s^{-1/\alpha}P_{t-s}((s^{1/\alpha}+|\cdot|)^{-\beta})(x)\,ds.
\end{align*}
By Lemma \ref{lem:estpa}, $P_{t-s}(1/(s^{1/\alpha}+|\cdot|)^{\beta})(x)\approx 1/(t^{1/\alpha}+|x|)^{\beta}$, and therefore 
\begin{align*}
|\nabla u^{(n)}(t,x)| &\lesssim \frac{1}{t^{1/\alpha}(t^{1/\alpha}+|x|)^{\beta}}\\
&\ \ \ \times\[1 + \int_0^{1-\varepsilon} \frac{\|U\|_\infty^{q}}{(1-r)^{2/\alpha}r^{(\alpha-1)/\alpha}}dr+ A_n\int_{1-\varepsilon}^1 \frac{\|U\|_\infty^{q}}{(1-r)^{1/\alpha}r}dr\]\\
&\lesssim\frac{1}{t^{1/\alpha}(t^{1/\alpha}+|x|)^{\beta}}\[1+\|U\|_\infty^{q}\(\ve^{-2/\alpha}+\varepsilon^{1-1/\alpha} A_n\)\].
\end{align*}
Consequently there is a constant $c$ not depending on $n$ and $\ve$ such that
\begin{align*}
A_n\le c\left(1+\|U\|_\infty^{q}\(\ve^{-2/\alpha}+\varepsilon^{1-1/\alpha} A_n\)\right).
\end{align*}
Thus, taking $\varepsilon$ small enough, we get $A_n\le c\(1+ \ve^{-2/\alpha}\|U\|_\infty^{q}\)/\(1-c\|U\|_\infty^{q}\ve^{1-1/\alpha}\)$, which proves \eqref{eq:A<oo}.
 In particular,  this implies
\begin{align}\nonumber
\left|U(x+z)-U(x)\right|&=\lim_{n\nearrow\infty}|u^{(n)}(1,x+z)-u^{(n)}(1,x)|\\
\label{eq:u-u}
&\leq |z|\sup_{0\leq v \leq 1}|\nabla u^{(n)}(1,x+vz)| \lesssim \frac{A|z|}{(1+|x|)^{\beta}}, 
\end{align}
whenever $|z|<1$. Let $e_1, \ldots e_n$ be the standard basis in $\RR^d$.
% For $h\not=0$, denote $h_k = he_k$.}
  Now, rewriting \eqref{eq:duh1} in a similar manner as it was done in \eqref{aux4}, for any $\varepsilon\in(0,1)$, $1\le k\le d$ and $h\in\RR-\{0\}$ we get
\begin{align}\label{eq:U-U}
&\frac{U(x+he_k)-U(x)}{|h|} = M\int_{\RR^d}\frac{p(1,x+he_k-y)-p(1,x-y)}{|h||y|^\beta}dy\\\nonumber
&+\alpha \int_0^{1-\varepsilon} \int_{\RR^d} r^{d-\beta}b\cdot \frac{\nabla_wp(1-r^\alpha,x+he_k,r w)- \nabla_wp(1-r^\alpha,x,r w)}{|h|}[ U(w)]^{q+1}\,dw\,dr\\\nonumber
&+\alpha \int_{1-\varepsilon}^1 \int_{\RR^d} r^{d-\beta} b\cdot \nabla_wp(1-r^\alpha,0,rw)\frac{[ U(w+(x+he_k)/r)]^{q+1}-[ U(w+x/r)]^{q+1}}{|h|}\,dw\,dr.  
\end{align}
By \eqref{eq:gradp}, we may pass with $|h|$ to zero under the first two integrals. Furthermore, \eqref{eq:u-u} together with \eqref{eq:pscaling} and \eqref{eq:estpa} let us bound the third integral for $|h|<1-\varepsilon$  by 
\begin{align*}
&C\|U\|^{q}\int_{1-\varepsilon}^1 \int_{\RR^d}r^{d-\beta} \frac{p(1-r^\alpha,0,rw)}{(1-r)^{1/\alpha}}\frac{A}{(1+|w+x/r|)^\beta}\,dw\,dr\\
&=A C\|U\|^{q}\int_{1-\varepsilon}^1 \int_{\RR^d}  \frac{r^{-\beta}}{(1-r)^{1/\alpha}}P_{r^{-\alpha}-1}\((1+|\cdot|)^{-\beta}\)(xr^{-1})\,dw\,dr\\
&\lesssim \frac A{(1+|x|)^{\beta}}\int_{1-\varepsilon}^1 (1-r)^{-1/\alpha} \,dr\leq \frac{\varepsilon^{1-1/\alpha}}{1-1/\alpha}  A.
\end{align*}
This gives us 
\begin{align*}
\limsup_{|h|\rightarrow0}\frac{U(x+he_k)-U(x)}{|h|} -\liminf_{|h|\rightarrow0}\frac{U(x+he_k)-U(x)}{|h|}\lesssim \varepsilon^{1-1/\alpha}A,
\end{align*}
for any $\varepsilon\in(0,1)$. Therefore the limit exists, and  is bounded by \eqref{eq:nablaU} due to \eqref{eq:u-u}. Furthermore, by \eqref{eq:U-U} for any $\varepsilon\in (0,1)$ we may express the gradient by 
\begin{align*}
\nabla U(x)& = \nabla P_1u_0(x)+\alpha \int_0^{1-\varepsilon} \int_{\RR^d} r^{d-\beta}\nabla \(b\cdot \nabla_wp(1-r^\alpha,x,r w)\) U^{q+1}(w)\,dw\,dr\\
&\ \ \ +\alpha \int_{1-\varepsilon}^1 \int_{\RR^d} r^{d-\beta} b\cdot \nabla_wp(1-r^\alpha,x,rw) U^q(w)\nabla U(w)\,dw\,dr,  
\end{align*}
which ensures its continuity by the estimates of $U$ and $\nabla U$.
\end{proof}
Below we summarize the results of this section.  
\begin{proof}[Proof of Theorem \ref{thm:sss}]
Let $u_0(x) = M|x|^{-\beta}$ with $M>0$ and $\beta \in (1,d)$. A solution $u(t,x)$ to the problem \eqref{eq:problem} exists by Theorem \ref{thm:existence} and Lemma \ref{lem:u_0ass}. By Lemma \ref{lem:selfsimilar} this solution is selfsimilar and there exists the function $U(x)$ such that $u(t,x) = t^{-\beta/\alpha} U(t^{-1/\alpha}x)$. The estimates \eqref{eq:sss2} follows by Proposition \ref{thm:Uest} and Corollary \ref{cor:Ulowest}. Proposition \ref{thm:asymp} yields \eqref{eq:sss3a} and \eqref{eq:sss3b}. Finally, {the regularity of $U$ and} \eqref{eq:sss4} follows by Proposition \ref{prop:gradUest}.
\end{proof}

\subsubsection*{Acknowledgement} We thank Grzegorz Karch for introducing us into the subject and for helpful discussions.

\end{document}